\documentclass[12pt,reqno]{amsart}
\usepackage{amsmath}
\usepackage{amsxtra,amssymb,amsthm,amsfonts,eufrak}
\usepackage{amsrefs}

\numberwithin{equation}{section}
\usepackage[svgnames]{xcolor}
\usepackage{graphicx}
\usepackage{float}
\usepackage{tikz}
\usepackage{tikz-cd}
\definecolor{DarkRed}{rgb}{0.9,0.7,0.3}
\tikzstyle{block} = [draw,rectangle,thick,minimum height=2em,minimum width=2em]
\DeclareMathOperator{\spn}{span}
\newcommand{\T}{\mathcal{T}}
\newcommand{\s}{\mathcal{S}}

\newcommand{\f}{\mathcal{F}}
\newcommand{\Q}{\mathcal{Q}}

\newcommand{\e}{\epsilon}

\renewcommand{\for}{\begin{eqnarray*}}
\newcommand{\mel}{\end{eqnarray*}}
\def\fr{\begin{align*}}

\newcommand{\ten}{\otimes}

\newcommand{\pl}{\hspace{.1cm}}

\newcommand{\ran}{\rangle}

\newcommand{\al}{\alpha}
\newcommand{\si}{\sigma}

\newcommand{\la}{\lambda}

\newcommand{\F}{{\mathbb F}}
\newcommand{\E}{{\mathcal E}}
\newcommand{\Z}{{\mathbb Z}}
\newcommand{\A}{{\mathcal A}}
\newcommand{\B}{{\mathcal{B}}}

\newcommand{\C}{{\mathbb C}}

\newcommand{\norm}[2]{\parallel \! #1 \! \parallel_{#2}}

\newtheorem{lemma}{Lemma}[section]
\newtheorem{prop}[lemma]{Proposition}
\newtheorem{theorem}[lemma]{Theorem}
\newtheorem{cor}[lemma]{Corollary}

\newtheorem{rem}[lemma]{Remark}

\newcommand{\re}{\begin{rem}\rm}
\newcommand{\mar}{\end{rem}}

\newcommand{\bra}[1]{\langle{#1}|}
\newcommand{\ket}[1]{|{#1}\rangle}
\newcommand{\ketbra}[1]{|{#1}\rangle\langle{#1}|}

\newcommand{\prf}{\begin{proof}[\bf Proof:]}

\newcommand{\xspace}{\hbox{\kern-2.5pt}}

\allowdisplaybreaks

\newtheorem{thmx}{Theorem}

\begin{document}
\title{Quantum Teleportation and super-dense coding in operator algebras}
\author{Li Gao}
\address{Department of Mathematics\\
University of Illinois, Urbana, IL 61801, USA} \email[Li Gao]{ligao3@illinois.edu}
\author{Samuel J. Harris}
\address{Department of Pure Mathematics, University of Waterloo, Waterloo, ON, N2L 3G1, Canada} \email[Samuel J. Harris]{sj2harri@uwaterloo.ca}
\author{Marius Junge}
\address{Department of Mathematics\\
University of Illinois, Urbana, IL 61801, USA} \email[Marius Junge]{mjunge@illinois.edu}
\begin{abstract} Let $\B_d$ be the unital $C^*$-algebra generated by the elements $u_{jk}, \,  0 \le i, j \le d-1$, satisfying the relations that $[u_{j,k}]$ is a unitary operator, and let $C^*(\F_{d^2})$ be the full group $C^*$-algebra of free group of $d^2$ generators. Based on the idea of teleportation and super-dense coding in quantum information theory, we exhibit the two $*$-isomorphisms $M_d(C^*(\F_{d^2}))\cong \B_d\rtimes \Z_d\rtimes \Z_d$ and $M_d(\B_d)\cong C^*(\F_{d^2})\rtimes \Z_d\rtimes \Z_d$, for certain actions of $\Z_d$. As an application, we show that for any $d,m\ge 2$ with $(d,m)\neq (2,2)$, the matrix-valued generalization of the (tensor product) quantum correlation set of $d$ inputs and $m$ outputs is not closed.
\end{abstract}
\maketitle

\setlength{\parindent}{4ex}
\section{Introduction}
\emph{Super-dense coding} \cite{superdense} and \emph{teleportation}\cite{teleportation},  devised by Bennett et al., are two fundamental protocols in quantum information theory. These two protocols together describe the fact that, with the assistance of quantum entanglement, quantum communication and classical communication are mutually convertible resources \cite{mother,resource}. Both protocols are examples of the extraordinary power of entanglement, and they demonstrate the fundamental role of non-local correlations in quantum information science. In this paper, we present a reformulation of super-dense coding and teleportation in terms of $C^*$-algebras isomorphisms. As an application, we show that the matrix-valued generalization of (tensor product) quantum correlation set of $d$ inputs and $m$ outputs is not closed.

Recall that Brown's noncommutative unitary $C^*$-algebra $\B_d$, defined in \cite{brown81}, is the universal $C^*$-algebra generated by elements $\{u_{jk}\}_{0\le j,k\le d-1}$ such that the
operator-valued matrix $[u_{jk}]_{jk}$ is a unitary operator. We also recall that the group $C^*$-algebra $C^*(\F_{d^2})$ is the universal $C^*$-algebra generated by $d^2$ unitaries. We show that the protocol maps of super-dense coding and teleportation translate into the following $C^*$-algebra isomorphisms.
\begin{thmx} With certain actions of $\Z_d$,
\begin{align*}M_d(C^*(\F_{d^2}))\cong \B_d\rtimes \Z_d\rtimes \Z_d\pl, M_d(\B_d)\cong C^*(\F_{d^2})\rtimes \Z_d\rtimes \Z_d\pl.
\end{align*} As a consequence, $\B_d$ (resp. $C^*(\F_{d^2})$) is a $C^*$-subalgebra of $M_d(C^*(\F_{d^2}))$ (resp. $M_d(\B_d)$) with a faithful conditional expectation onto it.\label{B}
\end{thmx}
\noindent The operator space perspective of super-dense coding and teleportation has been studied in \cite{JP16}. In particular, by \cite[Corollary~1.2 \& Theorem~1.3]{JP16}, the trace class $S_1^d$ and $l_1$-sequence space $l_1^{d^2}$, equipped with their natural operator space structure $S_1^d=(M_d)^*$ and $l_1^{d^2}=(l_\infty^{d^2})^*$, embed into certain matrix levels of each other via complete isometries, i.e.,
\begin{align}S_1^d\hookrightarrow M_d(l_1^{d^2})\pl,\pl l_1^{d^2}\hookrightarrow M_d(S_1^d) \pl.\label{11}\end{align}
Note that $C^*(\F_{d^2})$ (resp. $\B_d$) is the $C^*$-envelope of $l_1^{d^2}$ (resp. $S_1^d$) using suitable unitizations. Theorem \ref{B} can be viewed as liftings of \eqref{11} to $C^*$-algebras. It provides explicit connections between the two universal $C^*$-algebras and relates to some recent results in \cite{harris}. Moreover, the analogous result between ``reduced'' algebras is also obtained.

The second part of this work is devoted to applications in quantum correlations. Quantum correlations are probabilistic correlations that arise from measurement on bipartite quantum systems. Recall that a projection valued measurement (PVM) with $m$ outputs is an
$m$-tuple $(p_a)_{a=1}^m$ of orthogonal projections on a Hilbert space $H$ such that
$\sum_{a}p_a = 1$. A correlation matrix of $d$ inputs and $m$ outputs is a probability density of the form \[\{P(a,b|x,y)\}_{1\le a, b\le m, 1\le x,y\le d }\in \mathbb{R}^{d^2m^2}\pl.\] A correlation matrix $P(a,b|x,y)$ is called \emph{spatial quantum} (or \emph{spatial}) if there are Hilbert spaces $H_A$ and $H_B$,
a unit vector $\ket{\psi} \in H_A \ten H_B $, PVMs $\{ p^x_a \}_{1\le a\le m}$, for $x=1,\cdots, d$ on $H_A$, and PVMs $\{ q^y_b \}_{1\le b\le m}$, for $y=1,\cdots, d$ on $H_B$, such that
\[ P(a,b|x,y)= \bra{\psi}p_a^x\ten q_b^y\ket{\psi}\pl.\]
Let $\Q_s(d,m)$ be the set of all spatial quantum correlations in $d$ inputs and $m$ outputs. Slofstra \cite{slofstra2} proved that there exist $d$ and $m$ with $(d,m) \neq (2,2)$ such that $\Q_s(d,m)$ is not closed. That is, there exists a quantum correlation matrix which is a limit of spatial quantum correlations, but which cannot be observed by tensor product measurements. Slofstra's argument uses certain universal embedding theorems in geometric group theory, and the number $d$ obtained is larger than $100$. More recently, K.J. Dykema, V.I. Paulsen and J. Prakash \cite{Paulsen} proved that $\Q_s(5,2)$ is not closed. It remains open whether, for any ``nontrivial'' size ($(d,m)\neq (2,2), d,m\ge 2$), the spatial correlation set $\Q_s(d,m)$ is not closed. The second main theorem of this paper obtains new the non-closeness results of matrix-valued generalization of quantum correlation sets of size smaller than $(5,2)$.
\begin{thmx}\label{A}
For any $d,m\ge 2$ with $(d,m)\neq (2,2)$, there exists $n \in \mathbb{N}$, with $n \leq 13$, such that the following matrix-valued quantum correlation set
\[\Q_s^n(d,m)=\left\{ \Big[V^*(p_a^x\ten q_b^y)V\Big]_{\overset{x,v}{a,b}}\pl \bigm| \pl \begin{array}{c}
\text{$H_A,H_B$ Hilbert spaces,} \\\text{$V:l_2^n\to H_A\ten H_B$ an isometry,}\\
\text{$(p_a^x)_{a=1}$ PVMs on $H_A$ for each $x=1,\cdots,d$}\\
\text{$(q_b^y)_{b=1}$ PVMs on $H_B$ for each $y=1,\cdots,d$}
\end{array} \right\}  \]
is not closed.
\end{thmx}\noindent Matrix-valued correlations are the outcomes of partial measurements, and from the $C^*$-algebra perspective, $Q_s^n(d,m)$ is natural generalization of $Q_s(d,m)$ obtained by replacing states by $M_n$-valued unital completely positive (UCP) maps. The study of the sets of matrix-valued quantum correlations led to the equivalence of Connes' embedding problem and a matrix-valued version of Tsirelson's problem (see \cite{J+,Fritz}).  Later, Ozawa \cite{Ozawa13} proved that the scalar version of the Tsirelson problem is equivalent to Connes' embedding problem.

Our construction is closely related to the phenomenon known as \emph{embezzlement of entanglement}, introduced by Van Dam and Hayden \cite{embezzlement}. Cleve, Liu and Paulsen showed in \cite{cleve17} that the protocol for embezzling entanglement corresponds to a state on the minimal tensor product $\B_d\ten_{min} \B_d$ that cannot be implemented as a vector state via tensor product representations. Their result is in parallel to Stolstra's result in the sense that the Brown algebra has analogues of Kirchberg's conjecture \cite{harris} and Tsirelson's problem \cite{harris16}. Our idea is to apply the $*$-isomorphisms in Theorem \ref{B} to translate the non-spatial correlation from $\B_d\ten_{min} \B_d$ to $M_d(C^*(\F_{d^2}))\ten_{min} M_d(C^*(\F_{d^2}))$, and then use group embeddings of free groups into free products of cyclic groups to obtain matrix-valued correlations. Our argument gives explicit non-spatial matrix-valued correlations. In particular, we show that $Q^5_s(3,2)$, $Q^3_s(4,2)$ and $Q^{13}_s(2,3)$ are not closed. Theorem \ref{A} follows easily from the non-closure of $Q^5_s(3,2)$ and $Q^{13}_s(2,3)$.

The main part of this paper is divided into two sections. Section 2 reviews the basics of the protocols of super-dense coding and teleportation, and gives the proof of Theorem A. Based on that section, we show the non-closure of the matrix-valued quantum correlation sets $Q^5_s(3,2)$, $Q^3_s(4,2)$ and $Q^{13}_s(2,3)$ in Section 3.

\section{Teleportation and super-dense coding}
We briefly review the basic protocols of teleportation and super-dense coding and refer to \cite{wilde} for their information-theoretic meaning. Let $M_d$ be the space of $d\times d$ complex matrices, and let $l_2^d$ be the $d$-dimensional complex Hilbert space. We use the bracket notation $\{\ket{j}\}_{0\le j\le d-1}$ for the standard basis of $l_2^d$ and denote by $\{e_{jk}\}_{0 \leq j,k \leq d-1}$ the standard matrix units of $M_d$ given by $e_{jk}=\ket{j}\bra{k}$. The maximally entangled state on $l_2^d\ten l_2^d$ is $\displaystyle\ket{\phi}=\frac{1}{\sqrt{d}}\sum_{0\le j\le d-1}\ket{j} \ket{j}$ and  its density matrix is $\displaystyle\phi=\frac{1}{d}\sum_{0\le j,k\le d-1}e_{jk}\ten e_{jk}$. The generalized Pauli matrices are given by
\[X\ket{j}=e^{\frac{2\pi i j}{d}}\ket{j}\pl, \pl Z\ket{j}=\ket{j+1}\pl, \, \forall\pl 0 \leq j \leq d-1.\]
In the definition of $Z$ and in the remainder of the paper, the addition of indices will be considered modulo $d$.
We introduce the operators $T_{j,k} :=X^{j}Z^{k}$ and vectors $\ket{\phi_{jk}}:=(T_{j,k} \ten 1)\ket{\phi}$. Note that $\{\ket{\phi_{jk}}\}_{0\le j,k\le d-1}$ is a set of maximally entangled vectors, and they form an orthonormal basis for $l_2^d\ten l_2^d$.

Mathematically, the protocol of quantum teleportation for a $d$-dimensional system can be expressed as the follows:
\begin{align}
& M_d  \longrightarrow  M_d\ten M_d\ten M_d  \longrightarrow   l_\infty^{d^2}(M_d) \longrightarrow  M_d \pl, \nonumber
\\ &\rho  \longmapsto \rho \ten \phi  \longmapsto  \frac{1}{d^2}\sum_{0\le j,k\le d-1}\ket{jk}\bra{jk}\ten T_{jk}^*\rho T_{jk} \longmapsto  \rho\pl. \label{tp}
\end{align}
Here $\rho$ can be thought of as the quantum state that the sender Alice send to the receiver Bob. In the protocol, Alice first performs a measurement according to the basis $\{\ket{\phi_{jk}}\}_{j,k}$ on the coupled system of the input $\rho$ and her part of the maximally entangled state $\phi$. She sends the outcome of her measurement, a classical signal of cardinality $d^2$, to Bob via some classical channel. Then Bob reproduces the state $\rho$ by doing a unitary operation on his part according to the information received from Alice. Here a key calculation (see \cite[Lemma 2.1]{JP16}) is that
\[\rho \ten \phi=\frac{1}{d^2}\sum_{0\le j,k,j',k'\le d-1}\ket{\phi_{jk}}\bra{\phi_{j'k'}}\ten T_{jk}^*\rho T_{j'k'}\pl.\]
 The second map of \eqref{tp}, which corresponds to the measurement performed by Alice, is the conditional expectation from $M_d\ten M_d$ onto the commutative subalgebra spanned by $\{\ketbra{\phi_{jk}}\}_{j,k}$. Bob's action is the third map, which is
\begin{align*}\sum_{0\le j,k\le d-1}p_{jk}\ket{jk}\bra{jk}\ten \rho_{jk} \mapsto  \sum_{0 \le j,k\le d-1}p_{jk}T_{jk}^* \rho_{jk}T_{jk}\pl.\end{align*}
Using the same notation, super-dense coding can be expressed as follows:
\begin{align}
& l_\infty^{d^2}  \longrightarrow  l_\infty^{d^2}\ten M_d\ten M_d  \longrightarrow   M_d\ten M_d \longrightarrow  l_\infty^{d^2} \pl, \nonumber
\\ &(p_{jk})  \mapsto \big(\sum_{jk}p_{jk} \ketbra{jk} \big)\ten \phi  \mapsto  \sum_{jk}p_{jk} \ketbra{\phi_{jk}} \mapsto  \sum_{jk}p_{jk} \ketbra{jk}\pl.\label{sd}
\end{align}
This time Alice wants to transmit a classical signal $(p_{jk})$, which is a probability distribution. She first applies the unitary $T_{jk}$ on her part of the maximally entangled state $\phi$ according to the signal $(p_{jk})$, and then sends her part of $\phi$ to Bob via some quantum channel. Now Bob has both parts of the (modified) entangled state, and can perfectly decode the classical signal $(p_{jk})$ via a measurement according to the basis $\{\ket{\phi_{jk}}\}_{j,k}$.

The completely bounded norms of above maps were calculated in \cite{JP16}. We refer to \cite{Ruan,pis-intro} for the basics of operator space theory. Recall that the natural operator space structures of $S_1^d$ and $l_1^d$ are given by the operator space duality $S_1^d=(M_d)^*$ and $l_1^d=(l_\infty^d)^*$, with norms given as follows: for $\sum_{j=0}^{d-1} a_j \otimes e_j \in M_n(\ell_1^d)$ and $\sum_{j,k=0}^{d-1} a_{jk} \otimes e_{jk} \in M_n(S_1^d)$, we have
\begin{align*}
&\left\|\sum_{j=0}^{d-1} a_j\ten e_j \right\|= \sup \left\{\left\|\sum_{j=0}^{d-1} a_j\ten b_j \right\| \pl: b_j \in B(H), \, \|b_j\| \le 1\pl \right\}\pl, \\
&\left\|\sum_{j,k=0}^{d-1} a_{jk}\ten e_{jk}\right\|=\sup \left\{\left\|\sum_{j,k=0}^{d-1} a_{jk}\ten b_{jk}\right\| \pl : b_{jk} \in B(H), \, \left\|\sum_{j,k} e_{jk}\ten b_{jk}\right\| \le 1\pl \right\}\pl,
\end{align*}
where $\{e_j\}_{j=0}^{d-1}$ is the standard basis of $l_1^d$, $\{e_{jk}\}_{j,k=0}^{d-1}$ is the set of standard matrix units for $M_d$, and $a_j,b_j,a_{jk}$ and $b_{jk}$ are $n\times n$ matrices for $0 \leq j,k \leq d-1$.

The following theorem is from \cite[Section 2]{JP16}.
\begin{theorem}\label{os}The following maps are completely contractive:
\begin{align*}
&\mathcal{S}_1: S_1^d \to M_{d}(l_1^{d^2})\pl , \pl \s_1(\rho)=\frac{1}{d}\sum_{j,k=0}^{d-1} T_{jk}^*\rho T_{jk} \ten \ket{jk}\bra{jk} \pl,\\
&\T_1: M_d (l_1^{d^2}) \to S_1^{d}\pl , \pl \T_1(\rho \ten \ketbra{jk})=\frac{1}{d}T_{jk}\rho T_{jk}^* \pl,\\
&\s_2: l_1^{d^2} \to M_d (S_1^{d})\pl , \pl \s_2\left(\sum_{j,k=0}^{d-1} p_{jk}\ketbra{jk}\right)= d\sum_{j,k=0}^{d-1} p_{jk}\ketbra{\phi_{j,d-k}} \pl,\\
&\mathcal{T}_2: M_d (S_1^d) \to l_1^{d^2}\pl , \pl \T_2(\rho)=\frac{1}{d}\sum_{j,k=0}^{d-1} tr(\rho \phi_{j,d-k}) \ketbra{jk}\pl.
\end{align*}
Moreover $\T_1\circ\s_1 =id_{S_1^d}$ and $\T_2\circ\s_2 =id_{l_1^{d^2}}$.  In particular, $\s_1$ and $\s_2$ are complete isometries.
\end{theorem}
\begin{rem}{\rm The above complete contractions differ from the trace preserving maps in \eqref{tp} and \eqref{sd} by a scaling constant $d$. This difference is because, for each of the maximally entangled vectors $\phi_{jk}$, we have
\[\norm{\phi_{jk}}{S_1^{d^2}}=1 \pl, \pl \norm{\phi_{jk}}{M_d(S_1^d)}=\frac{1}{d}.\]
We also have flipped the indices $(j,k)\to (j,d-k)$ in $\s_2$ and $\T_2$; however, it is clear that our protocol is equivalent to the original protocol.}
\end{rem}

Our candidates for a $C^*$-algebraic analogue of teleportation and super-dense coding are the ``smallest'' $C^*$-algebras containing $S_1^d$ and $l_1^{d^2}$ respectively. Recall that a (concrete) unital operator space $E$ is closed subspace of a $C^*$-algebra containing the identity. The $C^*$-envelope $C_{env}^*(E)$ of a unital operator space $E$ is the unique $C^*$-algebra $C_{env}^*(E)$ equipped with a unital complete isometry $\iota: E\to C_{env}^*(E)$ satisfying the following property: for any unital complete isometry $j:E\to B(H)$, there exists a unique surjective $*$-homomorphism $\pi:C^*(j(E)) \to C_{env}^*(E)$ such that $\pi\circ j=\iota$, where $C^*(j(E))$ is the $C^*$-subalgebra of $B(H)$ generated by the image $j(E)$.

Recall that the (full) group $C^*$-algebra $C^*(\F_d)$ is the universal $C^*$-algebra generated by $d$ unitaries. The noncommutative unitary $C^*$-algebra $\B_d$ is defined to be the universal $C^*$-algebra generated by $\{u_{jk}\}_{0\le j,k\le d-1}$ such that $U:=\sum_{0\le j,k\le d-1}e_{jk}\ten u_{jk}$ is  unitary in $M_d(\B_d)$.

\begin{prop}\label{env} Let $\{g_j\}_{j=0}^{d-1}$ be the generators of $C^*(\F_d)$, and let $\{u_{jk}\}_{j,k=0}^{d-1}$ be the generators of $\B_d$. Define the operator spaces and unitalization \begin{align*}&\mathcal{X}_d:= \spn ( \{g_j\}_{j=0}^{d-1})\subseteq C^*(\F_d)\pl,\pl \tilde{\mathcal{X}}_d=\spn ( 1\cup\mathcal{X}_d)\subseteq C^*(\F_d)\pl; \\ &\mathcal{Y}_d:= \spn (  \{u_{jk}\}_{j,k=0}^{d-1})\subseteq \B_d \pl, \pl \tilde{\mathcal{Y}}_d=\spn ( 1\cup\mathcal{Y}_d)\subseteq \B_d\pl.\end{align*}
Then:
\begin{enumerate}
\item[i)]$\mathcal{X}_d \cong l_1^d$ completely isometrically and $C^*_{env}(\tilde{\mathcal{X}}_d)\cong C^*(\F_d)$.
\item[ii)]$\mathcal{Y}_d\cong S_1^d$ completely isometrically and $C^*_{env}(\tilde{\mathcal{Y}}_d)\cong \B_d$.
\end{enumerate}
\end{prop}
\begin{proof} The complete isometry in i) can also be found in \cite[Theorem 8.12]{pis-intro}; we include the proof for completeness. Let $H$ be an infinite dimensional Hilbert space and $a_j,a_{jk}$ be matrices in $M_n$ for $0 \leq j,k \leq d-1$. For an element $\sum_{j=0}^{d-1} a_j \otimes e_j \in M_n(l_1^d)$, we have
\begin{align*}
\left\|\sum_{j=0}^{d-1} a_j\ten e_j\right\|&=\sup \left\{\left\|\sum_{j=0}^{d-1} a_j\ten b_j\right\| \pl: b_j \in B(H), \, \|b_j\| \leq 1\right\} \\
&=\sup \left\{ \left\|\sum_{j=0}^{d-1} a_j\ten b_j \right\| \pl: b_j \text{ unitary in } B(H) \right\} \\
&=\left\| \sum_{j=0}^{d-1} a_j\ten g_j \right\|_{M_n(C^*(\F_d))}.
\end{align*}
Similarly, if $\sum_{j,k=0}^{d-1} a_{jk} \otimes e_{jk}$ is an element of $M_n(S_1^d)$, then
\begin{align*}
\left\| \sum_{j,k=0}^{d-1} a_{jk}\ten e_{jk} \right\|
&=\sup \left\{ \left\|\sum_{j,k=0}^{d-1} a_{jk}\ten b_{jk} \right\|: b_{jk} \in B(H), \, \left\| \sum_{j,k=0}^{d-1} b_{jk} \otimes e_{jk} \right\|_{M_n(B(H))} \leq 1\right\} \\
&=\sup \left\{ \left\|\sum_{j,k=0}^{d-1} a_{jk}\ten b_{jk} \right\|: \sum_{j,k=0}^{d-1} e_{jk}\ten b_{jk} \pl\text{unitary in}\pl M_d(B(H)) \right\} \\
&=\left\|\sum_{j,k=0}^{d-1} a_{jk}\ten b_{jk}\right\|_{M_n(\B_d)}. \\
\end{align*}
Thus, the maps
\begin{gather*}j_1:l_1^d \to \mathcal{X}_d\subset C^*(\F_d)\pl,
\pl\pl\pl  j_1(e_j)=g_j \pl, \\
 j_2:S_1^d \to \mathcal{Y}_d\subset\B_d\pl, \pl\pl \pl   j_2(e_{jk})=u_{jk} \pl,\end{gather*}
are complete isometries.

We now show that $C_{env}^*(\tilde{\mathcal{X}}_d)\cong C^*(\mathbb{F}_d)$. Since there is a unital completely isometric inclusion $\tilde{\mathcal{X}}_d \subseteq C^*(\mathbb{F}_d)$, by definition of the $C^*$-envelope, there is a surjective $*$-homomorphism $\Psi: C^*(\mathbb{F}_d) \to C_{env}^*(\tilde{\mathcal{X}}_d)$ such that $\Psi(g_j)=e_j$ for all $0 \leq j \leq d-1$.  Since each $g_j$ is unitary in $C^*(\mathbb{F}_d)$, each $e_j$ is unitary in $C_{env}^*(\tilde{\mathcal{X}}_d)$.  On the other hand, assume that $C^*(\F_d)\subseteq B(H)$ is a faithful representation, for some Hilbert space $H$. By Wittstock's extension theorem \cite{wittstock}, the unital complete isometry $\Phi:\tilde{\mathcal{X}}_d \to C^*(\F_d)\subseteq B(H)$ given by $\Phi(e_j)=g_j$ extends to a unital completely contractive (hence completely positive) map from $C^*_{env}(\tilde{\mathcal{X}}_d)$ to $B(H)$, which we will also denote by $\Phi$. Choose a minimal Stinespring dilation $\Phi(\cdot)=V\pi(\cdot)V^*$ for $\Phi$ on some Hilbert space $H_{\pi}$.  We may write $H_{\pi}=ran(V)\oplus ran(V)^{\perp}$. With respect to this decomposition,
\begin{align*}
\pi(e_j)=\left[\begin{array}{cc} \Phi(e_j) &*\\ *&*
\end{array}\right]=\left[\begin{array}{cc} g_j & * \\ * & * \end{array}\right]\pl.
\end{align*}
Since $e_j$ is unitary in $C^*_{env}(\tilde{\mathcal{X}}_d)$, $\pi(e_j)$ must be unitary in $B(H_{\pi})$.  But the $(1,1)$ entry of $\pi(e_j)$ is unitary as well, so the $(1,2)$ and $(2,1)$ entries must be $0$.  Therefore,
\begin{align*}
\pi(e_j)=\left[\begin{array}{cc} g_j &0\\0&*
\end{array}\right]=\left[\begin{array}{cc} \Phi(e_j) & 0 \\ 0 & * \end{array}\right]\pl.
\end{align*}
Thus, $\Phi$ is multiplicative on the generating set $\{e_j\}_{j=0}^{d-1}$ for $C^*_{env}(\tilde{\mathcal{X}}_d)$, so $\Phi$ must be a $*$-homomorphism. Moreover, $\Phi$ is surjective onto $C^*(\F_d)$ because the set $\{g_j\}_{0 \leq j \leq d-1}$ generates $C^*(\F_d)$.  Then $\Phi \circ \Psi(g_j)=g_j$ and $\Psi \circ \Phi(e_j)=e_j$ for all $j$.  It follows that $\Phi \circ \Psi=\text{id}_{C^*(\F_d)}$ and $\Psi \circ \Phi=\text{id}_{C^*_{env}(\tilde{\mathcal{X}}_d)}$. Hence, $C^*_{env}(\tilde{\mathcal{X}}_d)$ is isomorphic to $C^*(\F_d)$. The proof that $C_{env}^*(\tilde{\mathcal{Y}}_d)=\B_d$ is similar (see \cite[Theorem 4.3]{harris}).
\end{proof}

The next lemma shows that the embeddings from Theorem \ref{os} can be extended to $*$-homomorphisms on the respective $C^*$-envelopes. We will denote these $*$-homomorphisms by $\s_1$ and $\s_2$, respectively.

\begin{lemma}\label{oa} Let $\displaystyle\{g_{lm}\}_{0\le l,m\le d-1}$ be the generators of $C^*(\F_{d^2})$, and let $\displaystyle\{u_{jk}\}_{0\le j,k\le d-1}$ be the generators of $\B_d$. Define $\s_1:\B_d \to M_d(C^*(\F_{d^2}))$ and $\s_2: C^*(\F_{d^2}) \to M_d(\B_d)$ by
\begin{align*}
&\s_1(u_{jk})=\frac{1}{d}\sum_{l,m=0}^{d-1} e^{-\frac{2\pi i (j-k)l}{d}}e_{j-m,k-m} \ten g_{lm} \pl, \\ &\s_2(g_{lm})=\sum_{j,k=0}^{d-1} e^{\frac{2\pi i (j-k)l}{d}}e_{j-m,k-m}\ten u_{jk} \pl.
\end{align*}
\end{lemma}
\begin{proof} Let $\displaystyle U=\sum_{j,k=0}^{d-1} u_{jk}\ten e_{jk} \in M_d(\B_d)$ be the fundamental unitary of $\B_d$. Note that 
\begin{align*}
\sum_{j,k}e_{jk}\ten \s_1(u_{jk})=\frac{1}{d}\sum_{j,k,l,m}e_{jk}\ten  e^{-\frac{2\pi i (j-k)l}{d}}e_{j-m,k-m} \ten g_{lm} =\sum_{l,m}g_{lm}\ten \ketbra{\phi_{-l,m}}
\end{align*}
 is a unitary in $M_d\ten M_d(C^*(\F_{d^2}))$. For the second map $\s_2$, 
 \begin{align*}
\s_2(g_{lm})=\sum_{j,k}e^{\frac{2\pi i (j-k)l}{d}}e_{j-m,k-m}\ten u_{jk} =(T_{l,-m}\ten 1)U(T_{l,-m}\ten 1)^*\pl
\end{align*}
is a unitary in $M_d(\B_d)$ for each $0 \leq l,m \leq d-1$. By the universal property of $C^*(\F_{d^2})$, $\s_2$ is a $*$-homomorphism.
\end{proof}

Moving towards a proof of Theorem \ref{B}, we consider two automorphisms $\alpha_1,\alpha_2$ of $\B_d$ given as follows:
\[\al_1(u_{jk})=e^{\frac{2\pi i (j-k)}{d}}u_{jk}\pl, \pl \al_2(u_{jk})=u_{j+m,k+m}, \, \forall \pl 0 \leq j,k \leq d-1,\]
where, in the definition of $\al_2$, the addition of indices is done modulo $d$. Both $\al_1$ and $\al_2$ have order $d$; that is, $\al_1^d=\al_2^d=\text{id}_{\B_d}$ and $\al_i^k \neq \text{id}_{\B_d}$ for all $1 \leq k \leq d-1$ and $i=1,2$. They give two actions of
the cyclic group $\Z_d$ on $\B_d$. We define the iterated crossed product $\B_d\rtimes_{\al_1} \Z_d \rtimes_{\al_2}\Z_d$ with $\al_2$ acting on the first $\Z_d$ via character action. Namely, $\B_d\rtimes_{\al_1} \Z_d \rtimes_{\al_2}\Z_d$ is the universal $C^*$-algebra generated by the algebra of all sums of the form
\[F=\sum_{l,m=0}^{d-1}A_lv^lw^m\pl, \pl A_l\in \B_d,\]
where the product and adjoint are extended from $\B_d$ to satisfy, for each $A \in \B_d$,
\begin{align}\label{rule} &vAv^{-1}=\al_1(A)\pl , \pl v^*=v^{-1}=v^{d-1}\pl,\pl vw=e^{\frac{2\pi i }{d}}wv\pl, \nonumber \\ &wAw^{-1}=\al_2(A)\pl, \pl w^*=w^{-1}=w^{d-1}\pl.\end{align}
The iterated reduced crossed product $\B_d\rtimes_{\al_1,r} \Z_d \rtimes_{\al_2,r} \Z_d$ is isomorphic to the $C^*$-subalgebra of $M_d\ten M_d\ten\B_d$ generated by the range of the map $\pi:\B_d \to M_d\ten M_d\ten \B_d$ given by
\[\pi(A)=\sum_{l,m=1}^{d-1} e_{ll} \ten e_{mm}\ten\al_{1}^{-l}\al_{2}^{-m}(A) \pl, \]
and by the unitaries $v=Z\ten X\ten 1$ and $w=1\ten Z\ten 1$, where $X$ and $Z$ are the generalized Pauli matrices in $M_d$. Because $\Z_d$ is amenable, the full crossed product $\B_d\rtimes_{\al_1} \Z_d \rtimes_{\al_2} \Z_d$ isomorphic to the reduced crossed product $\B_d\rtimes_{\al_1,r} \Z_d \rtimes_{\al_2,r} \Z_d$ via the canonical quotient map \cite[Theorem 4.2.6]{brown08}.

We also define two automorphisms $\beta_1,\beta_2$ on $C^*(\F_{d^2})$ by
\[\beta_1(g_{jk})=g_{j+1,k}\pl, \pl \beta_2(g_{jk})=g_{j,k-1}\pl, \pl 0\le j,k\le d-1 \pl.\]
The iterated crossed product $C^*(\F_{d^2})\rtimes_{\beta_1}\Z_d\rtimes_{\beta_2}\Z_d$ is defined in a similar manner to \eqref{rule}. We are now in a position to prove Theorem \ref{B}.

\begin{theorem}\label{main} Let $\al_1,\al_2,\beta_1,\beta_2$ be the automorphisms given above. Then
 \[\B_d\rtimes_{\al_1} \Z_d \rtimes_{\al_2} \Z_d \cong M_d(C^*(\F_{d^2}))\pl, \text{ and } \pl C^*(\F_{d^2})\rtimes_{\beta_1} \Z_d \rtimes_{\beta_2} \Z_d \cong M_d(\B_d)\pl.\]
\end{theorem}
\begin{proof}Let $\s_1:\B_d \to M_d(C^*(\F_{d^2}))$ and $\s_2:C^*(\F_{d^2})  \to M_d(\B_d)$ be the $*$-homomorphisms from Lemma \ref{oa}. Then $(\s_1,X\ten 1)$ is a covariant representation of the $C^*$-dynamical system $(\B_d,\al_1,\Z_d)$.  By the universal property of $\B_d \rtimes_{\al_1} \Z_d$, this covariant representation induces a canonical $*$-homomorphism $\s_1': \B_d\rtimes_{\al_1} \Z_d\to M_d(C^*(\F_{d^2}))$ such that, for $A \in \B_d$ and the generator $v$ of $\Z_d$,
\[\s_1'(A)=\s_1(A) \pl, \pl \s_1'(v)=X \otimes 1 \pl.\]
Moreover, $(\s_1', Z\ten 1)$ is a covariant representation of $(\B_d\rtimes_{\al_1} \Z_d, \al_2, \Z_d)$, so it induces a canonical $*$-homomorphism $\tilde{\s_1} : \B_d\rtimes_{\al_1} \Z_d\rtimes_{\al_2} \Z_d\to M_d(C^*(\F_{d^2}))$ such that,
for $A\in \B_d$ and the generators $v,w$ of the two copies of $\Z_d$,
\[\tilde{\s_1}(A)=\s_1(A)\pl , \pl \tilde{\s_1}(v)=X\ten 1 \pl , \pl \tilde{\s_1}(w)=Z\ten 1\pl.\]
One can see that $\tilde{\s_1}$ is surjective since $\s_1(\B_d) \cup (M_d \ten \C 1)$ generates $M_d(C^*(\F_{d^2}))$. Now, consider the $*$-homomorphism $id_{M_d}\ten\s_2: M_d(C^*(\F_{d^2}))\to M_d\ten M_d\ten\B_d$. The range of $id_{M_d}\ten\s_2$ is generated by $(id_{M_d}\ten\s_2)\circ \s_1 (\B_d)$ and $M_d\ten \C 1\ten \C 1$. Note that for $0 \leq j,k \leq d-1$,
\begin{align*}
(id_{M_d}\ten\s_2)\circ \s_1(u_{jk})&=\frac{1}{d}\sum_{l,m}e^{-\frac{2\pi i (j-k)l}{d}}e_{j-m,k-m} \ten \s_2(g_{lm})
\\&=\frac{1}{d}\sum_{l,m,a,b}e^{-\frac{2\pi i (j-k)l}{d}}e_{j-m,k-m} \ten e^{\frac{2\pi i (a-b)l}{d}}e_{a-m,b-m}\ten u_{ab}
\\&=\sum_{l,m,n}e_{j-m,k-m} \ten e_{j+n-m,k+n-m}\ten u_{j+n,k+n}
\\&=\frac{1}{d}\sum_{m,n}e_{j-m,k-m} \ten e_{j+n-m,k+n-m}\ten e^{-\frac{2\pi i (j-k)l}{d}}\al_{1}^{l}\al_{2}^{n}(u_{jk})
\\&=\sum_{l,n}\ketbra{\phi_{-l,-n}}\ten \al_{1}^{l}\al_{2}^{n}(u_{jk})\pl.
\end{align*}
 Let $V\in M_d\ten M_d$ be the unitary given by $V(\ket{jk})=e^{-\frac{2\pi i jk}{d}}\ket{\phi_{jk}}$ for each $0 \leq j,k \leq d-1$. Then
\begin{align}&(V^*\ten 1)(id_{M_d}\ten\s_2)(A)(V\ten 1)= \sum_{0\le j, k\le d-1}  e_{jj}\ten e_{kk} \ten\al_{1}^{-j}\al_{2}^{-k}(A)\pl , \nonumber\\ &V^*(X\ten 1 )V=Z\ten X\pl , \pl V^*(Z\ten 1)V=1\ten Z. \pl \label{ue}\end{align}
In particular, $(V^* \otimes 1)(M_d \ten \s_2(C^*(\F_{d^2})))(V \otimes 1)=\B_d \rtimes_{\alpha_1,r} \Z_d \rtimes_{\alpha_2,r} \Z_d$. It follows that the map $(V^* \otimes 1)[(id_{M_d}\ten\s_2)\circ\tilde{\s_1}(\cdot)] (V \otimes 1)$ is the canonical quotient map from the full crossed product to the reduced crossed product, and must be a $*$-isomorphism. Therefore, $\tilde{\s_1}$ is injective, so that $\tilde{\s_1}$ is also a $*$-isomorphism.

The argument for the second isomorphism is similar. Using covariant representations, we obtain the surjective $*$-homomorphism \begin{align*}&\tilde{\s_2}: C^*(\F_{d^2})\rtimes_{\beta_1}\Z_d\rtimes_{\beta_2}\Z_d \to M_d(\B_d)\pl,\\
&\tilde{\s_2}(B)=\s_2(B)\pl , \pl \tilde{\s_2}(v)=X\ten 1 \pl , \pl \tilde{\s_2}(w)=Z\ten 1\pl.
\end{align*} Note that for each $0 \leq l,m \leq d-1$,
\begin{align*}
(id_{M_d}\ten\s_1)\circ \s_2(g_{lm})&=\frac{1}{d}\sum_{j,k}e^{\frac{2\pi i (j-k)l}{d}}e_{j-m,k-m} \ten \s_1(u_{jk})
\\&=\frac{1}{d}\sum_{j,k,a,b}e^{\frac{2\pi i (j-k)l}{d}}e_{j-m,k-m} \ten e^{-\frac{2\pi i (j-k)a}{d}}e_{j-b,k-b}\ten g_{ab}
\\&=\frac{1}{d}\sum_{j,k,a,b}\ketbra{\phi_{l-a,-m+b}}\ten \beta_1^{a-l}\beta_2^{m-b}(g_{lm})\pl.
\\&=\frac{1}{d}\sum_{a,b}\ketbra{\phi_{a,b}}\ten \beta_1^{-a}\beta_2^{-b}(g_{lm})\pl.
\end{align*}
Conjugating by the same unitary $V\ten 1$ as in \eqref{ue}, we obtain the canonical quotient map from the full crossed product to the reduced crossed product, which is a $*$-isomorphism.  Hence, $\tilde{\s}_2$ is a $*$-isomorphism.
\end{proof}

\begin{cor}
There exist unital completely positive maps $\T_1: M_d(C^*(\F_{d^2}))\to \B_d$ and $\T_2: M_d(\B_d)\to C^*(\F_{d^2})$ such that $\T_1\circ \s_1= id_{\B_d}$ and $\T_2\circ \s_2= id_{C^*(\F_{d^2})}$. As a consequence, $\s_1$ and $\s_2$ are injective $*$-homomorphisms.
\end{cor}
\begin{proof}
There are natural conditional expectations $\E_1:\B_d\rtimes_{\al_1} \Z_d \rtimes_{\al_2} \Z_d\to \B_d$ and ${\E_2: C^*(\F_{d^2})\rtimes_{\beta_1} \Z_d \rtimes_{\beta_2} \Z_d \to C^*(\F_{d^2})}$ given by
\[\E_1 \left(\sum_{l,m=0}^{d-1} A_{lm}v^lw^m \right)=A_{00}\pl \text{ and } \pl \E_2 \left(\sum_{l,m=0}^{d-1} B_{lm}v^lw^m \right)=B_{00}\pl.\]
Define \[\T_1=\E_1\circ\tilde{\s_1}^{-1}\pl \text{ and }\pl \T_2=\E_2\circ\tilde{\s_2}^{-1}\pl,\]
where $\tilde{\s_1}$ and $\tilde{\s_2}$ are the isomorphisms from Theorem \ref{main}. It readily follows that $\T_1 \circ \s_1=\text{id}_{\B_d}$ and $T_2 \circ \s_2=\text{id}_{C^*(\F_{d^2})}$.
\end{proof}

Isomorphisms analogous to those in Theorem \ref{main} can be obtained for the reduced $C^*$-algebras associated to $C^*(\F_{d^2})$ and $\B_d$. Let $C(\mathbb{T})$ be the $C^*$-algebra of continuous functions on the unit circle. McClanahan \cite{McClanahan} proved that $M_d(\B_d)$ is $*$-isomorphic to the free product $M_d*_\C C(\mathbb{T})$ by the map
\[u_{jk}\to \sum_{0\le l\le d-1}e_{lj}ue_{kl}\pl, \pl e_{jk}\to e_{jk}\pl, \]
where $u$ is the unitary $u(z)=z$ for all $z\in \mathbb{T}$ which generates $C(\mathbb{T})$ (we refer to \cite{freeprob} for information about free products and reduced free products of $C^*$-algebras). Moreover, $\B_d$ is isomorphic to the relative commutant $M_d^c$ in $M_d * C(\mathbb{T})$.  The reduced Brown algebra $\B_d^{red}$ is defined as the relative commutant $M_d^c$ in the reduced free product $M_d*_\C^{red}C(\mathbb{T})$, where the reduced free product is taken with respect to the unique trace $tr$ on $M_d$ and the canonical trace $\tau$ on $C(\mathbb{T})$ given by
\[tr([a_{jk}])=\frac{1}{d}\sum_{j}a_{jj}\pl \text{ and } \tau(f)=\frac{1}{2\pi}\int^{2\pi}_{0} f(e^{it})dt \pl.\]
We recall that the reduced group $C^*$-algebra of $C^*_{red}(\F_{d})$ is the $C^*$-algebra generated by the range of the left-regular representation
\[ \la: \F_{d}\to B(l_2(\F_{d})) \pl , \text{ where }\pl \la(g)\ket{h}=\ket{gh}\pl,\pl \forall\pl g,h\in\F_{d}.\]

Our argument for the reduced case relies on the following lemma.
\begin{lemma}Let $(\A,\phi)$ and $(\B,\psi)$ be two unital $C^*$-algebras equipped with states $\phi$ and $\psi$, respectively. Denote by $\pi_\phi$ (resp. $\pi_\psi$) the GNS-representation of $\phi$ (resp. $\psi$). Suppose that $\al:\A \to \B$ is a $*$-isomorphism satisfying $\phi=\psi\circ\al$. Then there exists a $*$-isomorphism $\al_\pi:\pi_\phi(\A)\to \pi_\psi(\B)$ such that $\al_\pi \circ\pi_\phi=\pi_\psi\circ\al$.\label{induced}
\end{lemma}
\begin{proof} Let $L_2(\A,\phi)$, $L_2(\B,\psi)$ be the Hilbert spaces of the GNS construction for $\phi$ and $\psi$, respectively. For any $a_1,a_2\in \A$, we have
\[\psi(\al(a_1)^*\al(a_2))=\psi(\al(a_1^*a_2))=\phi(a_1^*a_2)\pl.\]
Thus, $\al$ induces a unitary operator $V_\al: L_2(\A,\phi)\to L_2(\B,\psi)$. Consider
\[\al_\pi(\cdot)=V_{\al} (\cdot) V_\al^*\]
as a $*$-isomorphism from $B(L_2(\A,\phi))$ onto $B(L_2(\B,\psi))$. Note that for every $a \in \A$,
\begin{align*}& \al_\pi(\pi_\phi(a))=V_\al \pi_\phi(a) V_\al^*= \pi_\psi(\al(a))\pl.  \qedhere \end{align*}
Since $\alpha:\A \to \B$ is a $*$-isomorphism, $\al_{\pi}:\pi_{\phi}(\A) \to \pi_{\psi}(\B)$ must be a $*$-isomorphism, completing the proof.
\end{proof}

Note that the reduced free product $M_d*_\C^{red}C(\mathbb{T})$ is isomorphic to the range of the GNS representation of the full free product $M_d*_\C C(\mathbb{T})$ with respect to free product trace $tr*\tau$. Similarly, the reduced group $C^*$-algebra $C^*_{red}(\F_{d^2})$ is isomorphic to the range of the GNS representation of $C^*(\F_{d^2})$ with respect to the canonical trace $\omega:C^*(\F_{d^2})$ given on finite sums by \[\omega\left(\sum_{\gamma \in \F_{d^2}} a_{\gamma} \gamma \right)=a_1.\]
By Lemma \ref{induced}, the actions $\al_1,\al_2$ on $\B_d$ and $\beta_1,\beta_2$ on $C^*(\F_{d^2})$ induce reduced versions of the actions because they preserve the traces involved. For simplicity, the reduced versions of the actions on $\B_d$ and $C^*(\F_{d^2})$ will also be denoted by $\al_1,\al_2,\beta_1$ and $\beta_2$.

\begin{cor}\label{red}Let $\al_1$ and $\al_2$ the the actions on $\B_d^{red}$ induced by the actions on $\B_d$, and let $\beta_1,\beta_2$ be the actions on $C^*_{red}(\F_{d^2})$ induced by the actions on $C^*(\F_{d^2})$. Then \[\B^{red}_d\rtimes_{\al_1}\Z_d\rtimes_{\al_2}\Z_d\cong M_d(C^*_{red}(\F_{d^2}))\pl \text{ and } \pl C^*_{red}(\F_{d^2}) \rtimes_{\beta_1}\Z_d\rtimes_{\beta_2}\Z_d\cong M_d(\B^{red}_d)  \pl
\pl.\]
\end{cor}
\begin{proof}We begin by proving the second isomorphism. Because $\beta_1$ and $\beta_2$ preserve the trace $\omega$, the (reduced) crossed product $C^*_{red}(\F_{d^2}) \rtimes_{\al_1}\Z_d\rtimes_{\al_2}\Z_d$ is the GNS representation of $C^*(\F_{d^2}) \rtimes_{\al_1}\Z_d\rtimes_{\al_2}\Z_d$ with respect to the natural extension of $\omega$ to $C^*(\F_{d^2}) \rtimes_{\beta_1} \Z_d \rtimes_{\beta_2} \Z_d$ given by
\[\widehat{\omega}\left(\sum_{l,m=0}^{d-1} A_{lm}v^lw^m \right)=\omega(A_{00})\pl.\]
Thus, it is sufficient to show that the isomorphism $\tilde{\s_2}$ from Theorem \ref{main} is trace-preserving.  That is to say, it is sufficient to show that ${(tr*\tau)\circ\tilde{\s_2}=\widehat{\omega}}$. Let $g=g_{j_1,k_1}^{\e_1}g_{j_2,k_2}^{\e_2}\cdots g_{j_n,k_n}^{\e_n}$ be a reduced word in $\F_{d^2}$, where $(j_a,k_a) \neq (j_b,k_b)$ for $a \neq b$, and $\e_1,\cdots,\e_n\in \Z \setminus \{0\}$. Recall that $T_{j,k}=X^jZ^k$, where $X$ and $Z$ are the generalized Pauli matrices in $M_d$. Let $\displaystyle U=\sum_{j,k}e_{jk}\ten u_{jk}$ be the fundamental unitary of $M_d(\B_d)$, and let
\[U_{j,k}:=(T_{j,k}\ten 1)^*U(T_{j,k}\ten 1)=\tilde{\s_2}(g_{j,k})\pl.\]
We observe that for $0 \leq a_1,a_2,b_1,b_2 \leq d-1$, we have \[T_{a_1,b_1}T_{a_2,b_2}^*=X^{a_1}Z^{b_1-b_2}X^{-a_2}=\exp \left(\frac{2\pi i a_2(b_1-b_2)}{d}\right) T_{a_1-a_2,b_1-b_2}.\] Then for $0 \leq j_{n+1},k_{n+1} \leq d$,
\begin{align*}\tilde{\s_2}(gv^{j_{n+1}}w^{k_{n+1}})&= U_{j_1,k_1}^{\e_1}U_{j_2,k_2}^{\e_2}\cdots U_{j_n,k_n}^{\e_n}(T_{j_{n+1},{k_{n+1}}}\ten 1)
\\
&=(T_{j_1,k_1} \otimes 1)^* U (T_{j_1,k_1} \otimes 1) \cdots (T_{j_n,k_n} \otimes 1)^* U (T_{j_n,k_n} \otimes 1) (T_{j_{n+1},k_{n+1}} \otimes 1) \\
&=\lambda \cdot (T_{j_1,k_1}\ten 1)^* U^{\e_1} (T_{j_2-j_1,k_2-k_1}\ten 1)^* \cdots U^{\e_n}(T_{j_{n+1}+j_n,k_{n+1}+k_n}\ten 1) \pl, \\
\end{align*}
where \[\lambda=\exp\left(\frac{2\pi ij_{n+1}(k_n+k_{n+1})\sum_{l=1}^{n-1}j_{l+1}(k_l-k_{l+1})}{d}\right).\]
Thus, by definition of $tr*\tau$, $\widetilde{\s}_2(gv^{j_{n+1}}w^{k_{n+1}})$ has trace zero. When $g=1$,
\[ \tilde{\s_2}(v^{j}w^{k})=T_{j,k}\ten 1,\]
and this element has trace $1$ if $(j,k)=(0,0)$ and trace $0$ otherwise. It follows that $(tr*\tau)\circ\tilde{\s_2}=\omega$, so that the second isomorphism follows from Lemma \ref{induced}. For the first isomorphism, we start with a matrix version of the full algebras:
\[ M_d\ten M_d(C^*(\F_{d^2}))\cong M_d(\B_d)\rtimes_{\al_1}\Z_d\rtimes_{\al_2}\Z_d \cong M_d*_\C C(\mathbb{T})\rtimes_{\al_1}\Z_d\rtimes_{\al_2}\Z_d \pl.\]
We want to show that the extension of free product trace $\widehat{tr*\tau}$ on $M_d*_{\C} C(\mathbb{T}) \rtimes_{\alpha_1} \Z_d \rtimes_{\alpha_2} \Z_d$ coincides with the product trace $tr\ten tr\ten \omega$ on $M_d \otimes M_d(C^*(\F_{d^2}))$.
Now, $M_d(\B_d)\rtimes_{\al_1}\Z_d\rtimes_{\al_2}\Z_d$ is spanned by elements of the form
\begin{align}\label{ele}U_{j_1k_1}^{\e_1}U_{j_2k_2}^{\e_2}\cdots U_{j_nk_n}^{\e_n}(X^{j_{n+1}}Z^{k_{n+1}}\ten 1)v^{j_{n+2}}w^{k_{n+2}}\pl, \end{align}
for $(j_1,k_1)\neq (j_2,k_2)\neq \cdots \neq (j_n,k_n)$ and nonzero $\e_1,\cdots,\e_n\in \Z$. The trace $\widehat{tr*\tau}$ of \eqref{ele} is $1$ if the word equals the identity (which happens when $n=j_{n+1}=j_{n+2}=k_{n+1}=k_{n+2}=0$), and $0$ otherwise. On the other hand, for each $0 \leq j,k \leq d-1$,
\begin{align*}id_{M_d}\ten\tilde{\s_1}(U_{jk})&=\frac{1}{d}\sum_{a,b,l,m}e^{\frac{2\pi i(a-b)j}{d}}e_{a+k,b+k}\ten e^{-\frac{2\pi i(a-b)l}{d}}e_{a-m,b-m}\ten g_{lm}\\&= \sum_{l,m}\ket{\phi_{j-l,m+k}}\bra{\phi_{j-l,m+k}}\ten g_{lm}
\\&= \sum_{l,m}\ket{\phi_{l,m}}\bra{\phi_{l,m}}\ten g_{j-l,m-k}\pl.\end{align*}
Thus, the isomorphism $id_{M_d}\ten\tilde{\s_1}$ sends  an element of the form \eqref{ele} to
\begin{align*}&\sum_{l,m}e^{2\pi i{j_{n+1}k_{n+2}}}\Big(\ket{\phi_{l,m}}\bra{\phi_{l,m}}(T_{j_{n+1},k_{n+1}}\ten T_{j_{n+2},k_{n+2}})\Big)\ten g_{j_1-l,m-k_1}^{\e_1}\cdots g_{j_n-l,m-k_n}^{\e_n}\\=&
\sum_{l,m}e^{2\pi i{j_{n+1}k_{n+2}}}\ket{\phi_{l,m}}\bra{\phi_{l-j_{n+1}-j_{n+2},m-k_{n+1}+k_{n+2}}}\ten g_{j_1-l,m-k_1}^{\e_1}\cdots g_{j_n-l,m-k_n}^{\e_n}\pl.\end{align*}
An element of this form is always of $0$ trace by $tr\ten tr \ten \omega$, because the word \\${g_{j_1-l,m-k_1}^{\e_1}\cdots g_{j_n-l,m-k_n}^{\e_n}}$ is reduced and non-trivial. When $n=0$, \[id_{M_d}\ten\tilde{\s_1}((T_{j_{1},k_{1}}\ten 1)v^{j_{2}}w^{k_{2}})=T_{j_{1},k_{1}}\ten T_{j_{2},k_{2}}\ten 1,\]
and in this case the trace is still preserved. By Lemma \ref{induced}, we have that
\[ M_d\ten M_d(C^*_{red}(\F_{d^2}))\cong M_d(\B_d^{red})\rtimes_{\al_1}\Z_d\rtimes_{\al_2}\Z_d\pl.\]
Note that this isomorphism maps $M_d\ten \C 1\ten \C 1$ to $M_d\ten \C1$. Therefore the first isomorphism follows from taking the relative commutant.
\end{proof}
For the corresponding von Neumann algebras, we have the following remark.
\begin{rem}{\rm Let $L(\F_{d^2})$ be the free group factor corresopnding to $\F_{d^2}$, i.e., it is the weak$^*$-closure of $C^*_{red}(\F_{d^2})$. Dykema \cite{Dykema93} proved the following formula:
\[M_{d}*_\C^{red}L(\F_{k})\cong L(\F_{d^2k})\ten M_d \pl, \pl  \forall \pl 1\le k\le\infty,d\ge 2\pl.\]
When $k=1$, $M_{d}(\overline{\B^{red}_d}^{w^*})\cong M_{d}*_\C^{red}L(\Z)\cong M_{d}(L(\F_{d^2}))$, which implies that the von Neumann algebra $\overline{\B^{red}_d}^{w^*}$ is isomorphic to $L(\F_{d^2})$. Note that the actions $\al_1,\al_2,\beta_1,\beta_2$ are all trace preserving. Then the two isomorphisms of Corollary \ref{red} merge when taking weak$^*$-closure. That is to say, we have the isomorphisms
\[M_d(L(\F_{d^2}))\cong L(\F_{d^2}) \rtimes_{\al_1}\Z_d\rtimes_{\al_2}\Z_d \cong L(\F_{d^2}) \rtimes_{\beta_1}\Z_d\rtimes_{\beta_2}\Z_d\pl.\]}
\end{rem}

\section{Matrix valued quantum correlation sets}
Quantum correlation sets under different assumptions correspond to different kinds of states on tensor products of $C^*$-algebras. We will consider a generalization of these correlation sets where the states are replaced by matrix-valued unital completely positive (UCP) maps (also known as matricial states). Let $\A$ and $\B$ be two unital $C^*$-algebras. We say that a UCP map $\Psi:\A \ten_{min}\B \to M_n(\C)$ is \emph{spatial} if there exist $*$-representations ${\pi_1: \A\to B(H_A)}$, ${\pi_2:\B\to B(H_B)}$ on some Hilbert spaces $H_A,H_B$ and an isometry $V:l_2^n\to H_A\ten H_B$ such that
\[\Psi(\cdot)= V^*(\pi_1\ten \pi_2)(\cdot)V \pl.\]
We say that $\Psi$ is \emph{finite dimensional} if $H_A$ and $H_B$ can be taken to be finite dimensional. We denote by ${S(\A\ten_{min}\B, M_n)}$ the set of all UCP maps from $\A\ten_{min}\B$ to $M_n$, ${Sp(\A\ten\B, M_n)}$ the set of \emph{spatial} UCP maps and ${Sf(\A\ten\B, M_n)}$ the set of \emph{finite dimensional} UCP maps. When $n=1$, we simply write $S(\A \otimes \B)$, $Sp(\A \otimes \B)$, and $Sf(\A \otimes \B)$ for these sets. It is clear that for any $\A, \B$ and $n\ge 1$,
\[Sf(\A\ten\B, M_n)\subseteq Sp(\A\ten\B, M_n)\subseteq S(\A\ten_{min}\B, M_n).\]
By a lemma of Bunce and Salinas \cite{buncesalinas}, $Sp(\A\ten\B, M_n)$ is dense in $S(\A\ten_{min}\B, M_n)$ in the point-norm topology.

Let $\Z_m$ be the finite cyclic group of order $m$. We denote by $(p_a)_{a=1}^m$ the standard basis of minimal projections in $l_\infty^m$, noting that $l_{\infty}^m \cong C^*(\Z_m)$ via the Fourier transform.  Similarly, we let $(p_a^x)_{a=1}^m$ be the standard basis of minimal projections for the $x$-th copy of $C^*(\Z_m)$ in $*_d C^*(\Z_m) \simeq C^*(*_d \Z_m)$ on the left side of the tensor product $C^*(*_d \Z_m) \otimes_{\min} C^*(*_d \Z_m)$. We let $(q_b^y)_{b=1}^m$ be the same basis for the $y$-th copy of $C^*(\Z_m)$ on the right side of the tensor product. We will consider the matrix-valued quantum correlation sets arising from the images of the various subsets of matricial states on the generators of the form $p_a^x \otimes q_b^y$.  In particular, we define
\begin{align*}&\mathcal{Q}_f^{n}(d,m):=\Big\{ [\Psi(p_a^x\ten q^y_b)]_{a,b,x,y} \pl | \pl \Psi \in Sf(C^*(*_d\Z_m) \ten C^*(*_d\Z_m),M_n)\Big\}\pl,\\
&\Q_s^{n}(d,m):=\Big\{ [\Psi(p_a^x\ten q^y_b)]_{a,b,x,y} \pl | \pl \Psi \in Sp(C^*(*_d\Z_m) \ten C^*(*_d\Z_m),M_n)\Big\}\pl,\\
&\Q^{n}(d,m):=\overline{\Q_s^{n}(d,m)}=\Big\{ [\Psi(p_a^x\ten q^y_b)]_{a,b,x,y} \pl | \pl \Psi \in S(C^*(*_d\Z_m) \ten_{min} C^*(*_d\Z_m),M_n)\Big\}\pl,\\
&\Q_c^{n}(d,m):=\Big\{ [\Psi(p_a^x\ten q^y_b)]_{a,b,x,y} \pl | \pl \Psi \in S(C^*(*_d\Z_m) \ten_{max} C^*(*_d\Z_m),M_n)\Big\}\pl.\end{align*}
Here $\Q_f^n$ is the set of correlations representable on finite dimensional Hilbert spaces; $\Q_s^n$ is the set of correlations representable on a tensor product system; its closure $\Q^{n}$ is given by restrictions of UCP maps on $C^*(*_d\Z_m) \ten_{min} C^*(*_d\Z_m)$; and $\Q_c^n$ is the set of all correlations arising from commuting measurement systems. It is easy to see the expression of $\Q_s^n$ is equivalent to the formulation in Theorem \ref{A}. When $n=1$, we will simply write $Q_f,Q_s,Q$ and $Q_c$ for the correlation sets. It was observed in \cite{PT15}
that these variations leads to the following hierarchy of correlation sets, for all $d,m,n$:
\begin{align}\label{hierarchy}\mathcal{Q}_f^{n}(d,m)\subseteq \mathcal{Q}_s^{n}(d,m)\subseteq \mathcal{Q}^{n}(d,m)\subseteq \mathcal{Q}_c^{n}(d,m)\pl.\end{align}
Tsirelson's problem asks whether $Q(d,m)=Q_c(d,m)$ for every $(d,m)$. It was established in \cite{J+,Fritz,Ozawa13} that Tsirelson's problem is equivalent to Connes' embedding problem. In particular, due to \cite[Theorem 36]{Ozawa13}, Tsirelson's problem admits an equivalent matricial problem, which asks whether, for any fixed $(d,m) \neq (2,2)$ with $d,m \geq 2$, the equality $\mathcal{Q}^{n}(d,m)= \mathcal{Q}_c^{n}(d,m)$ holds for all $n \in \mathbb{N}$. Thus, it is also interesting to exhibit separations of the form $\mathcal{Q}_s^{n}(d,m)\neq \mathcal{Q}^{n}(d,m)$ for matrix level correlation sets when $(d,m)$ is small and the separation between $\mathcal{Q}_s(d,m)$ and $\mathcal{Q}(d,m)$ is not known.

We introduce the \emph{free correlation sets} as an analogue of quantum correlations for $C^*(\F_d)$. Motivated by the hierarchy of quantum correlation sets, we define the sets
\begin{align*}&\mathcal{F}_f^{n}(d):=\Big\{ [\Psi(g_j\ten g_k)]_{j,k} \pl | \pl \Psi \in Sf(C^*(\F_d) \ten C^*(\F_d),M_n)\Big\}\pl,\\
&\f_s^{n}(d):=\Big\{ [\Psi(g_j\ten g_k)]_{j,k} \pl | \pl \Psi \in Sp(C^*(\F_d) \ten C^*(\F_d),M_n)\Big\}\pl,\\
&\f^{n}(d):=\overline{\f_s^{n}(d)}=\Big\{ [\Psi(g_j\ten g_k)]_{j,k} \pl | \pl \Psi \in S(C^*(\F_d) \ten_{min} C^*(\F_d),M_n)\Big\}\pl,\\
&\f_c^{n}(d):=\Big\{ [\Psi(g_j\ten g_k)]_{j,k} \pl | \pl \Psi \in S(C^*(\F_d) \ten_{max} C^*(\F_d),M_n)\Big\}\pl.\end{align*}
Tsirelson's problem also has an equivalent version in terms of these sets; namely, it is equivalent to determining whether $\f(d)=\f_c(d)$ for every $d \geq 2$ (see \cite[Theorem 29]{Ozawa13}). The analogue of the above sets for the Brown algebra, called \emph{unitary correlation sets}, were studied in \cite{harris16}.

Cleve, Liu and Paulsen \cite{cleve17} proved that there exists a state $\psi$ on $\B_d\ten_{min}\B_d$ such that
\begin{align}&\psi(u_{j0}\ten u_{k0})=\frac{1}{\sqrt{d}}\delta_{jk}\pl, \pl 0\le j,k\le d-1, \label{state}
\end{align}
and any state satisfying \eqref{state} is not spatial. Their argument was based on work of van Dam and Hayden \cite{embezzlement} regarding the embezzlement of entanglement. We translate this non-spatial correlation from $\B_d\ten_{min}\B_d$ to an $M_d$-valued ucp map on $C^*(\F_{d^2})\ten_{min} C^*(\F_{d^2})$ by the isomorphisms obtained in Section 2. For simplicity, we will only consider the case $d=2$. Nevertheless, the argument works for all $d\ge 2$.

Let $\s_1:\B_2\to M_2(C^*(\F_4))$ and $\s_2:C^*(\F_4)\to M_2(\B_2)$ be the embeddings given in Lemma \ref{oa}. For notational convenience, we will let $\s_2^{(2)}=\text{id}_{M_2} \otimes \s_2:M_2(C^*(\F_4)) \to M_4(\B_2)$. Let $$\rho=\frac{1}{2}(e_{00} \otimes e_{00}+e_{30}+e_{30}+e_{03} \otimes e_{03}+e_{33} \otimes e_{33}),$$
which is a density matrix in $M_4 \otimes M_4$. Then we can factor the state $\psi$ as
\[\psi=(\rho\ten\psi)\circ(\s_2^{(2)}\ten \s_2^{(2)})\circ (\s_1\ten \s_1)\pl.\]
Here the choice of $\rho$ is not unique; we have simply chosen one of minimal rank. This factorization yields a state $\widetilde{\psi}$ on $M_2(C^*(\F_{4})) \ten_{min}M_2(C^*(\F_{4}))$ given by
\[\tilde{\psi}=(\rho\ten \psi) \circ \big(\s_2^{(2)} \ten \s_2^{(2)}\big) \pl.\]
In other words, the following diagram commutes:
\begin{center}
$\begin{tikzcd}[row sep=large]
\B_2 \otimes_{\min} \B_2 \ar{d}[swap]{\psi} \ar{r}{\s_1 \otimes \s_1} & M_2(C^*(\F_4)) \otimes_{\min} M_2(C^*(\F_4)) \ar{r}{\s_2^{(2)} \otimes \s_2^{(2)}} \ar{dl}[swap]{\widetilde{\psi}} & M_4(\B_2) \otimes_{\min} M_4(\B_2) \ar{d}{\simeq} \\
\mathbb{C} & & \ar{ll}[swap]{\rho \otimes \psi} M_4 \otimes M_4 \otimes (\B_2 \otimes_{\min} \B_2)
\end{tikzcd}$
\end{center}
Since $\psi=\tilde{\psi} \circ (\s_1 \ten \s_1)$, the state $\tilde{\psi}$ is not spatial, otherwise $\psi$ would be spatial. The state $\tilde{\psi}$ can be transformed into a UCP map from $C^*(\F_{4}) \ten_{min}C^*(\F_{4})$ to $M_2$, which on the generators is given as follows:
\begin{align}
&\Psi(g_{00}\ten g_{00})=\Psi(g_{10}\ten g_{10})=\left[\begin{array}{cc}\frac{1}{\sqrt{2}}& 0\\\frac{1}{\sqrt{2}}&0
\end{array}\right]\pl, \Psi(g_{10}\ten g_{00})=\Psi(g_{00}\ten g_{10})=\left[\begin{array}{cc}\frac{1}{\sqrt{2}}& 0\\-\frac{1}{\sqrt{2}}&0
\end{array}\right]\pl, \nonumber\\
&\Psi(g_{01}\ten g_{01})=\Psi(g_{11}\ten g_{11})=\left[\begin{array}{cc}0& \frac{1}{\sqrt{2}}\\0&\frac{1}{\sqrt{2}}
\end{array}\right]\pl, \Psi(g_{01}\ten g_{11})=\Psi(g_{11}\ten g_{01})=\left[\begin{array}{cc}0& -\frac{1}{\sqrt{2}}\\0&\frac{1}{\sqrt{2}}
\end{array}\right]\pl, \nonumber\\
&\Psi(g_{jk}\ten g_{lm})=\left[\begin{array}{cc}0& 0\\0&0
\end{array}\right]\pl\text{if}\pl k\neq m\pl. \label{value}
\end{align} $\Psi$ cannot be a spatial UCP map from $C^*(\F_4)\ten_{min}C^*(\F_4)$ to $M_2$, otherwise $\tilde{\psi}$ would be.
Motivated by the above observation, we have the following theorem.
\begin{theorem}\label{m2}Let $g_{00},g_{10},g_{01},g_{11}$ be the generators of $C^*(\F_4)$. There do not exist $*$-homomorphisms $\pi_1, \pi_2: C^*(\F_4)\to B(H)$ and orthonormal vectors $\ket{h_0},\ket{h_1}\in H\ten H$ such that, for $0 \leq j,k \leq 1$,
\begin{align}
 &\bra{h_0}\pi_1(g_{j0})\ten \pi_2(g_{l0})\ket{h_0}=
\frac{1}{\sqrt{2}} \pl,
\pl &\bra{h_1}\pi_1(g_{j0})\ten \pi_2(g_{k0})\ket{h_0}=\frac{(-1)^{j-k}}{\sqrt{2}}\pl, \label{1}
 \\&\bra{h_1}\pi_1(g_{j1})\ten \pi_2(g_{k1})\ket{h_1}=\frac{1}{\sqrt{2}}\pl, &\bra{h_0}\pi_1(g_{j1})\ten \pi_2(g_{k1})\ket{h_1}=\frac{(-1)^{j-k}}{\sqrt{2}}\pl.\label{2}
\end{align}
\end{theorem}
\begin{proof}
Suppose we have a setting described by \eqref{1} and \eqref{2}.
By summing up the corresponding equations in \eqref{1}, we have
\begin{align}
&(\bra{h_0}+\bra{h_1})\pi_1(g_{00})\ten \pi_2(g_{00})\ket{h_0}=(\bra{h_0}+\bra{h_1})\pi_1(g_{10})\ten \pi_2(g_{10})\ket{h_0}=\sqrt{2}\pl, \label{relation1}\\
&(\bra{h_0}-\bra{h_1})\pi_1(g_{10})\ten \pi_2(g_{00})\ket{h_0}=(\bra{h_0}-\bra{h_1})\pi_1(g_{00})\ten \pi_2(g_{10})\ket{h_0}=\sqrt{2}\pl, \label{relation}
\end{align}
Set $\ket{\eta_0}=\frac{1}{\sqrt{2}}(\ket{h_0}+\ket{h_1})$ and $\ket{\eta_1}=\frac{1}{\sqrt{2}}(\ket{h_0}-\ket{h_1})$. Then $\ket{\eta_0}$ and $\ket{\eta_1}$ are orthonormal because $\ket{h_0}$ and $\ket{h_1}$ are orthonormal. Equations \eqref{relation1} and \eqref{relation} imply that
\begin{align}
&\pi_1(g_{00})\ten \pi_2(g_{00})\ket{h_0}=\pi_1(g_{10})\ten \pi_2(g_{10})\ket{h_0}=\ket{\eta_0}\pl, \label{4}\\
&\pi_1(g_{10})\ten \pi_2(g_{00})\ket{h_0}=\pi_1(g_{00})\ten \pi_2(g_{10})\ket{h_0}=\ket{\eta_1}\pl. \label{3}
\end{align}
Similarly, by \eqref{2}, we obtain \begin{align*}
&\pi_1(g_{01})\ten \pi_2(g_{01})\ket{h_1}=\pi_1(g_{11})\ten \pi_2(g_{11})\ket{h_1}=\ket{\eta_0}\pl, \nonumber\\
&\pi_1(g_{01})\ten \pi_2(g_{11})\ket{h_1}=\pi_1(g_{01})\ten \pi_2(g_{11})\ket{h_1}=\ket{\eta_1}\pl.
\end{align*}
Thus, the vectors $\ket{h_0},\ket{h_1},\ket{\eta_0}$ and $\ket{\eta_1}$ can be converted to each other by local operations (i.e., tensor products of unitaries), which implies that $\ket{h_0},\ket{h_1}$ and $\ket{\eta_0}$ have the same Schmidt coefficients. On the other hand,
consider the unitaries $X_1=\pi_1(g_{10}g_{00}^*)$ and $X_2=\pi_2(g_{10}g_{00}^*)$. Equations \eqref{4} and \eqref{3} show that
\[X_1\ten 1\ket{h_0}=1\ten X_2\ket{h_0}=\ket{h_0}\pl \text{ and } \pl X_1\ten 1\ket{h_1}=1\ten X_2\ket{h_1}=-\ket{h_1}\pl.\]
Then \begin{align*}\ket{h_0}\in P_1H\ten P_2H\pl,\ket{h_0}\in (P_1H)^\perp\ten (P_2H)^\perp\pl,\end{align*}
where $P_1$ (resp. $P_2$) is the spectral projection of $X_1$ (resp. $X_2$) corresponding to the eigenvalue $1$. In this situation, if the largest Schmidt coefficient of $\ket{h_0}$ and $\ket{h_1}$ is $\la_0>0$, then the largest Schmidt coefficient of $\ket{\eta_0}$ is at most $\frac{1}{\sqrt{2}}\la_0$, which leads to a contradiction since this coefficient must be $\la_0$.
\end{proof}

\begin{rem} {\rm \label{app} It is clear from $\eqref{value}$ that the relations \eqref{1}--\eqref{2} described in Theorem \ref{m2} can be represented by a non-spatial UCP map $\Psi:C^*(\F_4)\ten_{min} C^*(\F_4)\to M_2$. This fact can also be observed directly by using approximate embezzlement of entangled states from \cite{embezzlement}. Mathematically, approximate embezzlement of entanglement describes the following fact: for each $n \geq 1$, there exists a finite dimensional Hilbert space $H_n$, a unit vector $\ket{\xi_n}\in H_n\ten H_n$, and unitary operators $U_n$ on $l_2^2\ten H_n$ and $V_n$ on $H_n\ten l_2^2$, such that
\begin{align}
\left\|(U_n\ten V_n)\ket{0} \ket{\xi_n} \ket{0}-\frac{1}{\sqrt{2}}(\ket{0} \ket{\xi_n}\ket{0}+\ket{1} \ket{\xi_n} \ket{1})\right\|_{l_2^2\ten H_n\ten H_n\ten l_2^2}\le \frac{1}{n} \pl. \label{em}\end{align}
We choose a sequence of $*$-representation $\pi_1^n,\pi_2^n:C^*(\F_4) \to M_2(B(H_n))$ such that
\begin{align*}&\pi_1^n(g_{00})=U_n, \, \pi_1^n(g_{10})=(X\ten 1)U_n\pl, \\ &\pi_1^n(g_{01})=(X\ten 1)U_n (Z\ten 1) \pl, \pi_1^n(g_{11})=U_n (Z\ten 1)\pl,\\
&\pi_2^n(g_{00})=V_n\pl,\, \pi_2^n(g_{10})=(X\ten 1)V_n\pl, \\&\pi_2^n(g_{01})=(X\ten 1)V_n (Z\ten 1) \pl, \,  \pi_2^n(g_{11})=V_n(Z\ten 1)\pl,
\end{align*}
where $X,Z$ are the Pauli matrices in $M_2$. Set $\ket{h^n_0}=\ket{0}\ket{\xi_n}\ket{0}$, $\ket{h^n_1}=\ket{1}\ket{\xi_n}\ket{1}$, and set $W_n:l_2^2\to l_2^2\ten H_n\ten H_n \ten l_2^2$ to be the isometry given by $W_n\ket{j}=\ket{h^n_j}$ for $j=0,1$. By \eqref{em},
we can choose $\Psi$ as a weak$^*$-limit point of the UCP maps
\[ \Psi_n(\cdot)=W_n^*(\pi_1^n\ten\pi_2^n)(\cdot)W_n\pl.\]}
\end{rem}
The next lemma is a digression of our discussion on the above non-spatial correlations.
\begin{lemma} \label{fact}i) Let $H$ be a Hilbert space. There exist unitaries $u_0,u_1,v_0,v_1$ on $H$ and orthonormal vectors $\ket{h_0},\ket{h_1}\in H\ten H$ satisfying
\begin{align}\bra{h_0}u_j\ten v_k\ket{h_0}=
\frac{1}{\sqrt{2}} \pl,
\pl \bra{h_1}u_j\ten v_k\ket{h_0}=\frac{(-1)^{j-k}}{\sqrt{2}}\pl, \pl \pl 0\le j,k\le 1 \pl \label{eq1}\end{align}
if and only if $H$ is infinite dimensional.\\
ii) For any Hilbert space $H$, there do not exist unitaries $u_0,u_1,u_2,v_0,v_1,v_2$ on $H$ and orthonormal vectors $\ket{h_0},\ket{h_1}\in H\ten H$ satisfying the equations \eqref{eq1} and in addition satisfying
\begin{align}u_2\ten v_2\ket{h_0}=\ket{h_1}\pl. \label{e2}
\end{align}
iii) There exists a UCP map $\Psi: C^*(\F_3)\ten_{min} C^*(\F_3)\to M_2$ such that
\begin{align}\label{free3}
\Psi(g_{0}\ten g_{0})=\left[\begin{array}{cc}*& *\\ 1 & *
\end{array}\right]\pl, \pl\Psi(g_{j}\ten g_{k})=\left[\begin{array}{cc}\frac{1}{\sqrt{2}}& *\\ \frac{(-1)^{j-k}}{\sqrt{2}}&*
\end{array}\right]\pl,\pl 1\le j,k\le 2 \pl.
\end{align}
Moreover any UCP map that satisfies \eqref{free3} is not spatial.
\end{lemma}
\begin{proof}i) As in the proof of Theorem \ref{m2}, equation \eqref{eq1} implies that
\begin{align}
&u_0\ten v_0\ket{h_0}=u_1\ten v_1\ket{h_0}=(\ket{h_0}+\ket{h_1})/\sqrt{2}\pl, \label{ue}
\\& u_0\ten v_1\ket{h_0}=u_1\ten v_0\ket{h_0}=(\ket{h_0}-\ket{h_1})/\sqrt{2},\nonumber
\end{align}
and moreover
\begin{align}\ket{h_0}\in P_1H\ten P_2H\pl,\ket{h_0}\in (P_1H)^\perp\ten (P_2H)^\perp\pl. \label{sm}\end{align}
where $P_1$ (resp. $P_2$) is the projection onto the eigenspace of $u_1u_0^*$ (resp. $v_1v_0^*$) corresponding to the eigenvalue $1$. When $H$ is finite dimensional, the Schmidt ranks of $\ket{h_0},\ket{h_1}$, and $\frac{1}{\sqrt{2}}(\ket{h_0}+\ket{h_1})$ are all finite and nonzero. Equation \eqref{sm} implies that the Schmidt rank of $\frac{1}{\sqrt{2}}(\ket{h_0}+\ket{h_1})$ is sum of the Schmidt ranks of $\ket{h_0}$ and $\ket{h_1}$. On the other hand, it follows from \eqref{ue} that $\ket{h_0}$ and $\frac{1}{\sqrt{2}}(\ket{h_0}+\ket{h_1})$ have the same Schmidt rank; this leads to a contradiction.

For the converse direction, we give an explicit contruction for when $H$ is infinite dimensional. We let $H=l_2(\Z)$ and let $\{\ket{j}\}_{ j\in \Z}$ be its standard basis. Define two unitaries
\[u\ket{j}=\ket{j+1}\pl, \si\ket{j}=\left\{\begin{aligned} -\ket{j} & \pl &j\ge 0\\  \ket{j} &\pl & j< 0\end{aligned}\right.\pl.\]
Choose orthonormal vectors $\ket{h_0}=\sum_{j<0}(\sqrt{2})^{j}\ket{j}\ket{j}$ and $\ket{h_1}=\ket{0}\ket{0}$ in $H\ten H$. We have
\begin{align*}&u\ten u\ket{h_0}=(\ket{h_0}+\ket{h_1})/{\sqrt{2}}\pl,\\ &\si\ten 1\ket{h_0}=1\ten \sigma\ket{h_0}=\ket{h_0}\pl , \pl \si\ten 1\ket{h_1}=1\ten \si\ket{h_1}=-\ket{h_1}\pl.\end{align*}
Then it is easy to see that setting $u_0=v_0=u$ and $u_1=v_1=\si u$ satisfies \eqref{eq1}.

For ii), the extra condition \eqref{e2} implies that $\ket{h_0},\ket{h_1}$ also have the same Schmidt coefficients. Combined with \eqref{ue} and
\eqref{sm}, this leads to the same contradiction as in Theorem \ref{m2}.

The UCP map $\Psi$ described in iii) can also be approximated using approximate embezzlement. With the same notation as in Remark \ref{app}, we define $\pi^n_1,\pi^n_2: C^*(\F_3)\to M_2(B(H_n))$ by
\begin{align}&\pi_1^n(g_{0})=Z\ten 1\pl,\pl  \pi_1^n(g_{1})=U_n\pl ,\pl \pi_1^n(g_{2})=(X\ten 1)U_n\pl, \nonumber\\
&\pi_2^n(g_{0})=Z\ten 1\pl , \pl \pi_2^n(g_{1})=V_n\pl, \pl \pi_2^n(g_{2})=(X\ten 1)V_n\pl. \label{appx}
\end{align}
and $\ket{h^n_0}=\ket{0}\ket{\xi_n}\ket{0}, \ket{h^n_1}=\ket{1}\ket{\xi_n}\ket{1}$. Note that $\pi_1^n(g_{0})\ten \pi_2^n(g_{0})\ket{h^n_0}=\ket{h^n_1}$ for all $n$. So the extra equation $\Psi(g_{0}\ten g_{0})=\left[\begin{array}{cc}*& *\\ 1 & *
\end{array}\right]$ is satisfied. Any such $\Psi$ is not spatial otherwise the map in ii) would be spatial.
\end{proof}
\begin{rem}{\rm \label{zzf}In the above lemma, the non-spatial UCP map $\Psi:C^*(\F_3)\ten_{min} C^*(\F_3)\to M_2$ induces a UCP map $\tilde{\Psi}:C^*(\Z_2*\Z_2*\Z)\ten C^*(\Z_2*\Z_2*\Z)\to M_2$. This is because in the approximation \eqref{appx}, the operators $\pi_1^n(g_{0}),\pi_2^n(g_{0}),\pi_1^n(g_{2}^{-1}g_{1})$ and $\pi_2^n(g_{2}^{-1}g_{1})$ are self-adjoint unitaries for every $n$. Let $\si_0,\si_1,g$ be the generator of $C^*(\Z_2*\Z_2*\Z)$ where $\si_0,\si_1$ are the self-adjoint unitaries. The relations in \eqref{free3} translate to
\begin{align}
&\tilde{\Psi}(\si_{0}\ten \si_{0})=\left[\begin{array}{cc}*& *\\ 1 & *
\end{array}\right]\pl, \pl\tilde{\Psi}(g\ten g)=\pl\tilde{\Psi}(\si_1g\ten \si_1g)=\left[\begin{array}{cc}\frac{1}{\sqrt{2}}& *\\ \frac{1}{\sqrt{2}}&*
\end{array}\right]\pl, \nonumber\\&\tilde{\Psi}(\si_1g\ten g)=\pl\tilde{\Psi}(g\ten \si_1g)=\left[\begin{array}{cc}\frac{1}{\sqrt{2}}& *\\ \frac{-1}{\sqrt{2}}&*
\end{array}\right] \pl.\label{zzfe}
\end{align}
}
\end{rem}

This UCP map leads to the following theorem, from which Theorem \ref{A} follows.  For the sets $Q_s^n(d,m)$, we note that part (iii) exhibits the smallest matrix size $n$ that we have obtained from our methods to show that $Q_s^n(d,m)$ is not closed.

\begin{theorem}\label{maintheorem}We have the following separations.
\begin{enumerate}
\item[i)]$\f_f^{2}(2)\neq \f_s^2(2)$;
\item[ii)]$\f_s^{2}(3)\neq \f^2(3)$;
\item[iii)]$\Q_s^{3}(4,2)\neq \Q^3(4,2)$;
\item[iv)]$\Q_s^{5}(3,2)\neq \Q^5(3,2)$;
\item[v)]$\Q_s^{13}(2,3)\neq \Q^{13}(2,3)$.
\end{enumerate}
\end{theorem}
\begin{proof}Let $g_1,g_2$ be the canonical generators of $C^*(\F_2)$ and let $g_0,g_1,g_2$ be the canonical generators of $C^*(\F_3)$. It is a direct consequence of Lemma \ref{fact} i) that the assignment
\begin{align*} \pl\Psi(g_{j}\ten g_{k})=\left[\begin{array}{cc}\frac{1}{\sqrt{2}}& *\\ \frac{(-1)^{j-k}}{\sqrt{2}}&*
\end{array}\right]\pl,\pl 1\le j,k\le 2 \pl
\end{align*}
represents a spatial $M_2$-valued correlation in $\f_s^{2}(2)$ which does not belong to $\f_f^2(2)$. Similarly, the separation $\f_s^{2}(3)\neq \f^2(3)$ follows from Lemma \ref{fact} ii) and iii).

For iii), we consider the embedding of $\Z_2*\Z_2*\Z$ into $*_4\Z_2$ given by
\[\si_0\mapsto \si_0\pl, \pl\si_1\mapsto \si_1\pl, \pl g\mapsto \si_2\si_3\pl,\]
where $\si_0,..., \si_3$ are generators of $\Z_2$ and $g$ is the generator of $\Z$. By \cite[Proposition 8.5]{pis-intro}, this embedding induces a $C^*$-algebra embedding $C^*(\Z_2*\Z_2*\Z)\hookrightarrow C^*(*_4\Z_2)$ and the non-spatial correlation in Remark \ref{zzf} extends to an $M_2$-valued ucp map on $C^*(*_4\Z_2)\ten_{\min} C^*(*_4\Z_2)$ satisfying the following:
\begin{align}&\tilde{\Psi}(\si_2\si_3\ten \si_2\si_3 )=\tilde{\Psi}(\si_1\si_2\si_3\ten \si_1\si_2\si_3)=\left[\begin{array}{cc}\frac{1}{\sqrt{2}}& *\\ \frac{1}{\sqrt{2}}&*
\end{array}\right]  \pl, \label{final1}\\
&\tilde{\Psi}(\si_1\si_2\si_3\ten \si_2\si_3 )=\tilde{\Psi}(\si_2\si_3\ten \si_1\si_2\si_3)= \left[\begin{array}{cc}\frac{1}{\sqrt{2}}& *\\ \frac{-1}{\sqrt{2}}&*
\end{array}\right]  \pl, \label{final2}\\
&\tilde{\Psi}(\si_0\ten \si_0)= \left[\begin{array}{cc}*& *\\ 1&*
\end{array}\right]  \pl.  \label{final4}
\end{align}
Let $\mathcal{P}=\spn \{1\} \cup \{\sigma_i\}_{i=0}^3$; we want to obtain a matrix-valued non-spatial correlation defined only on the operator system $\mathcal{P} \otimes \mathcal{P}$ in $C^*(*_4 \Z_2) \otimes_{\min} C^*(*_4 \Z_2)$, and not on products of elements of $\mathcal{P} \otimes \mathcal{P}$.  To accomplish this, we add extra dimensions to the output space. Let
$\tilde{\Phi}(\cdot)=V^*\pi(\cdot)V$
be a Stinespring dilation, where $\pi$ is a $*$-representation of ${C^*(*_4\Z_2) \ten_{min}C^*(*_4\Z_2)}$ on some Hilbert space $H$ and $V:l_2\to H$ is an isometry. Define unit vectors
\begin{align*}&\ket{h_0}:=V\ket{0}\pl, \pl \ket{h_1}:=V\ket{1},\pl \ket{h_2}:=\pi(\si_3\ten\si_3)\ket{h_0} \pl,
\end{align*}
and the operator \[W:l_2^3\to H\pl , \pl W\ket{j}=\ket{h_j}\pl, \pl 0\le j\le 2 \pl.\]
Then $\Psi(\cdot)=W^*\pi(\cdot)W$ gives an $M_3$-valued, possibly non-unital completely positive map. Let $\ket{\eta_0}=\frac{1}{\sqrt{2}}(\ket{h_0}+\ket{h_1})$ and $\ket{\eta_1}=\frac{1}{\sqrt{2}}(\ket{h_0}-\ket{h_1})$. We note that
$$\pi(\sigma_2 \sigma_3 \otimes \sigma_2 \sigma_3)\ket{h_0}=\pi(\si_1 \si_2\si_3 \otimes \si_1\si_2\si_3)\ket{h_0}=\ket{\eta_0},$$
while
\begin{align}
\pi(\sigma_1 \otimes 1)\ket{\eta_0}=\pi(1\otimes \sigma_1)\ket{\eta_0}=\ket{\eta_1} \text{ and } \pi(\sigma_1 \otimes \sigma_1)\ket{\eta_0}=\ket{\eta_0}. \label{onesidedsi1s}
\end{align}
Considering the actions of $\pi(\sigma_1 \otimes 1)$, $\pi(1 \otimes \sigma_1)$ and $\pi(\sigma_1 \otimes \sigma_1)$, we may write
\begin{align}
& 1=\frac{1}{2} (\bra{h_0}\pi(\sigma_1 \otimes 1)\ket{h_0}-\bra{h_1}\pi(\sigma_1 \otimes 1)\ket{h_1})-i\text{Im}\bra{h_1}\pi(\sigma_1 \otimes 1)\ket{h_0}, \label{leftsidedsi1}\\
& 1=\frac{1}{2} (\bra{h_0}\pi(1 \otimes \si_1)\ket{h_0}-\bra{h_1}\pi(1 \otimes \si_1)\ket{h_1})-i\text{Im}\bra{h_1}\pi(1 \otimes \si_1)\ket{h_0}, \label{rightsidedsi1}\\
& 1=\frac{1}{2} (\bra{h_0}\pi(\sigma_1 \otimes \sigma_1)\ket{h_0}+\bra{h_1} \pi(\sigma_1 \otimes \sigma_1) \ket{h_1})+\text{Re} \bra{h_1} \pi(\sigma_1 \otimes \sigma_1) \ket{h_0}. \label{doublesidedsi1}
\end{align}
In Equations \eqref{leftsidedsi1} and \eqref{rightsidedsi1}, since the first two summands are real, the last summand, being purely imaginary, must be zero. Then the first two summands average to $1$ and each must belong to $[-1,1]$, since $\pi(\sigma_1 \otimes 1)$ and $\pi(1 \otimes \sigma_1)$ are self-adjoint unitaries.  By convexity, this forces
\begin{align}
& 1=\bra{h_0}\pi(\sigma_1 \otimes 1)\ket{h_0}=\bra{h_0}\pi(1 \otimes \sigma_1)\ket{h_0}, \pl \\
& 1=-\bra{h_1}\pi(\sigma_1 \otimes 1)\ket{h_0}=-\bra{h_1}\pi(1 \otimes \sigma_1)\ket{h_0}.
\end{align}
In other words, $\pi(\sigma_1 \otimes 1)\ket{h_0}=\pi(1 \otimes \sigma_1)\ket{h_0}=\ket{h_1}$. Thus, $\pi(\sigma_1 \otimes \sigma_1)\ket{h_j}=\ket{h_j}$ for $j=0,1$. Since $\bra{h_1}{h_0}\ran=0$, the upper-left $2 \times 2$ block of $\Psi$ on each generator must be a contraction. Using the fact that $\Psi(\sigma_i \otimes \sigma_j)$, $\Psi(1 \otimes \sigma_i)$ and $\Psi(\sigma_i \otimes 1)$ are self-adjoint for all $0 \leq i \leq 3$, equations \eqref{final1}--\eqref{doublesidedsi1} show that
\begin{align}
&\Psi(\si_0 \otimes \si_0)=\begin{bmatrix} 0 & 1 & * \\ 1 & 0 & * \\ * & * & * \end{bmatrix}, \, \Psi(\sigma_2 \otimes \sigma_2)=\frac{1}{\sqrt{2}}\begin{bmatrix} * & * & 1 \\ * & * & 1 \\ 1 & 1 & * \end{bmatrix}, \pl \Psi(\si_3 \otimes \si_3)=\begin{bmatrix} * & * & 1 \\ * & * & * \\ 1 & * & * \end{bmatrix} \label{qs42nonspatial1}\\
&\Psi(\si_1 \otimes \si_1)=\begin{bmatrix} 1 & 0 & * \\ 0 & 1 & * \\ * & * & * \end{bmatrix}, \, \Psi(\si_1 \otimes 1)=\begin{bmatrix} 1 & 0 & * \\ 0 & -1 & * \\ * & * & * \end{bmatrix}, \pl \Psi(1 \otimes \si_1)=\begin{bmatrix} 1 & 0 & * \\ 0 & -1 & * \\ * & * & * \end{bmatrix}, \label{qs42nonspatial2}\pl
\end{align}

Then any CP map $\Theta:C^*(*_4\Z_2) \ten_{min}C^*(*_4\Z_2) \to M_3$ satisfying equations \eqref{qs42nonspatial1} and \eqref{qs42nonspatial2} cannot be spatial. Thus, we have shown that $\Psi|_{\mathcal{P}\ten \mathcal{P}}$ is a non-spatial correlation in the non-unital context. We replace $\Psi(\cdot)$ by a UCP map $\Psi'$ given as a point-weak$^*$-cluster point of the net of CP maps $(\Psi(1)+\e)^{-\frac{1}{2}} \Psi(\cdot) (\Psi(1)+\e)^{-\frac{1}{2}}$ for $\e>0$. Then the restriction $\Psi'|_{\mathcal{P}\ten \mathcal{P}}$ gives a $M_3$-valued ucp map, which is not spatial since $\Psi$ is not spatial.  Hence, we obtain a correlation in $\Q^3(4,2)$ which is not in $\Q_s^3(4,2)$.

For iv), we consider the following group embedding  of $\Z_2*\Z_2*\Z$ into $*_3\Z_2$:
\[\si_0\mapsto \si_0\pl, \pl \si_1\mapsto \si_1\pl , \pl g\mapsto \si_2\si_0\si_1\si_2 \pl.\]
Denote by $\omega$ the word $\si_2\si_0\si_1\si_2$. Then the non-spatial correlation in Remark \ref{zzf} extends to $C^*(*_3\Z_2)\ten C^*(*_3\Z_2)$ as follows: \begin{align}&\tilde{\Phi}(\omega\ten \omega)=\tilde{\Phi}(\si_1\omega\ten \si_1\omega)=\left[\begin{array}{cc}\frac{1}{\sqrt{2}}& *\\ \frac{1}{\sqrt{2}}&*
\end{array}\right]  \pl, \\
&\tilde{\Phi}(\si_1\omega\ten \omega )=\tilde{\Phi}(\omega\ten \si_1\omega)= \left[\begin{array}{cc}\frac{1}{\sqrt{2}}& *\\ \frac{-1}{\sqrt{2}}&*
\end{array}\right]  \pl, \\
&\tilde{\Phi}(\si_0\ten \si_0)= \left[\begin{array}{cc}*& *\\ 1&*
\end{array}\right]  \pl,
\end{align}
Let $\tilde{\Phi}(\cdot)=V_0^*\pi_0(\cdot)V_0$ be a Stinespring dilation of $\tilde{\Phi}$, and let $\ket{k_j}=V_0\ket{j}$ for $0\le j\le 1$. In order to express the above non-spatial correlation on the generators, we define three intermediate vectors
\begin{align*}&\ket{k_2}:=\pi_0(\si_2\ten \si_2)\ket{k_0}\pl, \ket{k_3}:=\pi_0(\si_1\ten \si_1)\ket{k_2}\pl, \\&\ket{k_4}:=\pi_0(\si_0\ten \si_0)\ket{k_3}\pl. \end{align*}
Then note that $\pi_0(\sigma_2 \otimes \sigma_2)\ket{k_4}=\pi_0(\omega \otimes \omega)\ket{k_0}=\frac{1}{\sqrt{2}}(\ket{k_0}+\ket{k_1})$. Thus, to describe $\tilde{\Phi}$ in terms of a matrix-valued correlation defined only on the generators, it suffices to use the vectors $\ket{k_j}$ for $0 \leq j \leq 4$.  In a similar manner to the argument for iii), we obtain a $M_5$-valued non-spatial correlation in $Q^5(3,2)$.

The same argument leads to a matrix-valued non-spatial correlation for $\Z_3*\Z_3$. Indeed, let $a,b$ be the generators of the two copies of $\Z_3$. We use the embedding $\F_3\to\Z_3*\Z_3$ given by
\[\si_0\mapsto aba, \pl \si_1\mapsto bab, \pl g \mapsto ab^{2}a^{2}b \pl.\]
It is not hard to see that these three words are free. Then the non-spatial correlation from Remark \ref{zzf} extends to a ucp map $\Gamma$ on $C^*(\Z_3*\Z_3) \otimes_{\min} C^*(\Z_3 * \Z_3)$ via
\begin{align}
&\Gamma(ab^2a^2b \otimes ab^2a^2b)=\Gamma(bab \cdot ab^2a^2b \otimes bab \cdot ab^2a^2b)=\begin{bmatrix} \frac{1}{\sqrt{2}} & * \\ \frac{1}{\sqrt{2}} & * \end{bmatrix} \pl, \\
&\Gamma(bab \cdot ab^2a^2b \otimes ab^2a^2b)=\Gamma(ab^2a^2b \otimes bab \cdot ab^2a^2b)=\begin{bmatrix} \frac{1}{\sqrt{2}} & * \\ -\frac{1}{\sqrt{2}} & * \end{bmatrix} \pl, \\
& \Gamma(aba \otimes aba)=\begin{bmatrix} * & * \\ 1 & * \end{bmatrix}.
\end{align}
We want a matrix-valued, non-spatial ucp map specified on $\mathcal{R} \otimes \mathcal{R}$, where $$\mathcal{R}=\spn \{1,a,b,a^*,b^*\}=\spn \{1,a,b,a^2,b^2\} \subseteq C^*(\Z_3 * \Z_3).$$
Let $\Gamma(\cdot)=V_1^* \pi_1 V_1$ be a Stinespring dilation for $\Gamma$ and $\ket{\zeta_j}=V_1 \ket{j}$ for $j=0,1$.  We define eleven intermediate vectors as follows:
\begin{align}
&\ket{\zeta_2}=\pi_1(a \otimes a)\ket{\zeta_0}, \, \label{gammafirstintermediatevectors} \ket{\zeta_3}=\pi_1(b \otimes b)\ket{\zeta_2}, \\
&\ket{\zeta_4}=\pi_1(b \otimes b)\ket{\zeta_0}, \, \ket{\zeta_5}=\pi_1(a^2 \otimes a^2)\ket{\zeta_4}, \, \ket{\zeta_6}=\pi_1(b^2 \otimes b^2)\ket{\zeta_5}. \pl
\end{align}
Let $\ket{\theta_0}=\frac{1}{\sqrt{2}}(\ket{\zeta_0}+\ket{\zeta_1})$ and $\ket{\theta_1}=\frac{1}{\sqrt{2}}(\ket{\zeta_0}-\ket{\zeta_1})$. It follows that
\begin{align}
\pi_1(a \otimes a)\ket{\zeta_6}=\pi_1(ab^2a^2b \otimes ab^2a^2b)\ket{\zeta_0}=\ket{\theta_0} \label{gammaong}
\end{align}
Define
\begin{align}
&\ket{\zeta_7}=\pi_1(1 \otimes b)\ket{\theta_0}, \, \ket{\zeta_8}=\pi_1(1 \otimes a)\ket{\zeta_7}, \pl \\
&\ket{\zeta_9}=\pi_1(b \otimes 1)\ket{\theta_0}, \, \ket{\zeta_{10}}=\pi_1(a \otimes 1)\ket{\zeta_9}, \\
&\ket{\zeta_{11}}=\pi_1(b \otimes b)\ket{\theta_0}, \, \ket{\zeta_{12}}=\pi_1(a \otimes a)\ket{\zeta_{11}}.
\end{align}
The definition of $\Gamma$ can be summarized in terms of the vectors via
\begin{align}
& \pi_1(a \otimes a)\ket{\zeta_6}=\pi_1(a \otimes a)\ket{\zeta_{11}}=\ket{\theta_0}, \\
& \pi_1(b \otimes 1)\ket{\theta_0}=\pi_1(1 \otimes b)\ket{\theta_0}=\ket{\theta_1}, \\
& \pi_1(a \otimes a)\ket{\zeta_3}=\ket{\zeta_1}. \label{gammalastintermediatevectors}
\end{align}
By the same argument as for iii) and iv), equations \eqref{gammafirstintermediatevectors} through \eqref{gammalastintermediatevectors} show that $Q_s^{13}(2,3)$ is not closed, which completes the proof.
\end{proof}
Note that for $(d,m)=(2,2)$, $\Z_2*\Z_2$ is an amenable group isomorphic to the semi-direct product $\Z\rtimes \Z_2$, and its irreducible representations are at most $2$-dimensional. Thus, our result is optimal in the sense that for $(d,m)=(2,2)$, $\Q_f^n(2,2)=\Q_s^n(2,2)=\Q^n(2,2)=\Q_c^n(2,2)$ for all $n$. We have shown that for all non-trivial sizes $(d,m) \neq (2,2)$, the spatial quantum correlation set in $d$ inputs and $m$ outputs is not closed at some matrix level. The question of whether scalar-valued quantum correlation sets are closed for sizes smaller than $(5,2)$ remains open.

\begin{rem}{\rm Let $G_1$ and $G_2$ be two finite groups. It is known that the product $G_1*G_2$ contains a copy of $\F_2$ if and only if $|G_1|+|G_2|\ge 5$ and $|G_1|,|G_2|>1$ (see, for example, \cite[p.~8]{serre}). Using this kind of group embedding, we know that there exists a $n$ such that the $M_n$-valued spatial correlation set for $C^*(\Z_2*\Z_3)$ is not closed (or for any $C^*(G_1*G_2)$).
}
\end{rem}
The main idea we used in the proof of Theorem \ref{maintheorem} is to reduce the length of the words in the definition of the UCP map by adding intermediate vectors. Conversely, we can reduce the matrix size for the UCP map by allowing correlations that use words with length greater than $1$. Such quantum correlations on words are called \emph{spatiotemporal correlations} in \cite{Fritz}.
In the context of spatiotemporal correlations, we have the following observation.
\begin{cor} Let $g_0,g_1,g_2$ be the canonical generators of $C^*(\F_3)$. There exists a state $\psi$ on ${C^*(\F_3)\ten_{min}C^*(\F_3)}$ satisfying
\begin{align}
 \psi(g_0\ten g_0)=0\pl,
 \pl\psi(g_j\ten g_k)=\frac{1}{\sqrt{2}}\pl,
\pl\psi(g_0g_j\ten g_0g_k)=\frac{(-1)^{j-k}}{\sqrt{2}}\pl, \pl 1\le j,k\le 2\pl.  \label{f3}
\end{align}
Moreover, any state satisfying \eqref{f3} is not spatial.
\end{cor}
\begin{proof}
Let $\Psi: C^*(\F_3)\ten_{min} C^*(\F_3)\to M_2$ be the UCP map from Lemma \eqref{fact} satisfying
\begin{align}\label{3g}
\Psi(g_{0}\ten g_{0})=\left[\begin{array}{cc}*& *\\ 1 & *
\end{array}\right]\pl, \pl\Psi(g_{j}\ten g_{k})=\left[\begin{array}{cc}\frac{1}{\sqrt{2}}& *\\ \frac{(-1)^{j-k}}{\sqrt{2}}&*
\end{array}\right]\pl,\pl 1\le j,k\le 2 \pl.
\end{align}
Because $\Psi$ is UCP and $g_0 \otimes g_0$ is a contraction, we know that $\Psi(g_{0}\ten g_{0})=\left[\begin{array}{cc}0& *\\ 1 & 0
\end{array}\right]$. Then it is easy to see that $\psi$ is given by the first diagonal entry of $\Psi$. Conversely, suppose such a state $\psi$ is spatially implemented by
\[\psi(\cdot)=\bra{h_0}(\pi_1\ten \pi_2)(\cdot)\ket{h_0}\pl.\]
Let $\ket{h_1}:=\pi_1(g_0)\ten \pi_2(g_0)\ket{h_0}$. Then $\ket{h_1}$ is orthogonal to $\ket{h_0}$ because \[\bra{h_0}h_1\ran=\bra{h_0}\pi_1(g_0)\ten \pi_2(g_0)\ket{h_0}=\psi(g_0\ten g_0)=0\pl.\] Then the vectors $\ket{h_0}$ and $\ket{h_1}$, together with the unitaries $\pi_1(g_j),\pi_2(g_j), 0\le j\le 2$ give exactly the scenario as in Lemma \ref{fact} ii), which is a contradiction.
\end{proof}
The point of the above corollary is that the non-spatial nature of the UCP map can be witnessed on words with length at most $2$, which span a finite dimensional subspace. Although the set of spatial states always forms a weak$^*$-dense subset of the state space of the minimal tensor product of two $C^*$-algebras, in general non-spatial states widely exist for the minimal tensor product. The following remark is due to Paulsen. Let $C[0,1]$ be the continuous functions on the unit interval. Then the minimal tensor $C[0,1]\ten_{min}C[0,1]\cong C([0,1]^2)$ is the continuous function on the square. A spatial state corresponds to a probability measure on $[0,1]^2$ that can be written as a (possibly infinite) convex combination of product probability measures. This is obviously not the case for the bilinear form
\[(f,g)\longrightarrow \int_{[0,1]}f(x)g(x)dx,  \pl\]
where $dx$ is the Lebesgue measure. In general, spatial states of $\A\ten \B$ correspond to nuclear operators from $\A$ to $\B^*$. From the above example, non-spatial states exist for $\A\ten \B$ whenever both $\A$ and $\B$ contain operators with continuous spectra, because non-spatial states are extendable to larger algebras. Nevertheless, non-spatial states witnessed on finite dimensional subsystem can only exist for noncommutative $\A$ and $\B$.

\begin{prop}Let $\A,\B$ be two $C^*$-algebras and $\A$ be commutative. Let $E\subset \A$ and $F\subset \B$ be two finite dimensional operator systems. For any state $\phi$ of $\A\ten_{\min} \B$, there exists a spatial state $\psi$ of $\A\ten_{\min} \B$ such that
$\tilde{\phi}|_{E\ten F}=\phi|_{E\ten F}$.
\end{prop}
\begin{proof}
The restriction $\phi|_{E\ten F}$ corresponds to a map $T_\phi:E\to F^*$
\[T_\phi(x)(y)=\phi(x\ten y) \pl, \forall  \pl x\in E, y\in F\pl.\]
Because $\A$ is commutative, $E\ten_{\min} F$ is isomorphic to the Banach space injective tensor product $E\ten_{\epsilon} F$ (see \cite{pis-intro}). Thus the integral norm of $T_\phi$ equals to $1$ and coincides with the nuclear norm since $E$ is finite dimensional (see \cite{pisier86} for definition of integral norm and nuclear norm). Then there exists some finite sequences $x_j^*\in E^*, y_j^*\in F^*$ and $\sum_j\la_j=1, \la_j>0$ such that
\[ \phi|_{E\ten F}=\sum_{j}\la_j x_j^*\ten y_j^* \pl. \]
and $\norm{x_j^*}{E^*}=\norm{y_j^*}{F^*}=1$ for any $j$. It is clear that $x_j^*$ and $y_j^*$ are positive because $\phi$ is. Let $\tilde{x}^*_j$ (resp. $\tilde{y}^*_j$) be the state of $\A$ (resp. $\B$) as an extension of $x_j^*$. Thus we have that
\[ \psi=\sum_{j}\la_j \tilde{x}^*_j\ten \tilde{y}^*_j \pl. \]
is a spatial state of $\A\ten_{min}\B$ which coincides with $\phi$ on $E\ten F$.
\end{proof}

{\noindent \bf Acknowledgements.}---We thank Vern I. Paulsen for explaining to us the idea of embezzlement of entanglement and for his remark on non-spatial states. We thank Narutaka Ozawa for comments on our old formulation of Theorem B. We also thank William Solfstra and Volker Scholz for helpful conversations. We thank Texas A$\&$M University for their kind hospitality during the conference ``Probabilistic and Algebraic Methods in Quantum information theory''. LG acknowledges support from NSF grant DMS-1700168, Illinois University Fellowship and Trjitzinsky Fellowship. MJ is partially supported by NSF
grant DMS-1501103. LG and MJ are grateful
to Institut Henri Poincar\'e for hospitality during participation to the trimester ``Analysis in quantum information''. 

\end{document}